\documentclass[12pt,english]{extarticle}
\usepackage[utf8]{inputenc}
\usepackage[a4paper]{geometry}
\usepackage{fancyhdr}
\usepackage{color}
\usepackage{amsmath}
\usepackage{amsthm}
\usepackage{amssymb}
\usepackage{bbm}
\usepackage{hyperref}

\makeatletter
\numberwithin{equation}{section}
\numberwithin{figure}{section}
\theoremstyle{plain}
\newtheorem{thm}{\protect\theoremname}[section]
\theoremstyle{plain}
\newtheorem{lem}[thm]{\protect\lemmaname}
\theoremstyle{plain}
\newtheorem{prop}[thm]{\protect\propositionname}
\theoremstyle{plain}
\newtheorem{cor}[thm]{\protect\corollaryname}
\theoremstyle{definition}
\newtheorem{rem}[thm]{\protect\remarkname}
\theoremstyle{definition}
\newtheorem*{assumption*}{\protect\assumptionname}

\usepackage[english]{babel}
\usepackage{amsthm}
\usepackage{amsfonts}

\usepackage{color}

\definecolor{rot}{rgb}{1.000,0.000,0.000}
\newcommand{\rb}[1]{\textcolor{rot}{\textbf{#1}}}

\usepackage{tikz}

\usepackage{tkz-graph}
\GraphInit[vstyle = Shade]
\tikzset{
  LabelStyle/.style = {rectangle, rounded corners, draw,
                        minimum width = 2em, 
                        text = black, font = \bfseries },
  VertexStyle/.style = {circle, draw, 
                                font =  \large\bfseries},
  EdgeStyle/.style = {->, bend left} }
\newcommand{\N}{\mathbb{N}}
\newcommand{\R}{\mathbb{R}}

\newcommand{\rC}{\mathcal{C}}
\newcommand{\rZ}{\mathcal{Z}}
\newcommand{\rX}{\mathcal{X}}
\newcommand{\rM}{\mathcal{M}}
\newcommand{\rP}{\mathcal{P}}

\newcommand{\oAs}{{\oA\!\!\!\;}^*}
\newcommand{\oAsh}{{\oA\!\!\,}^\star}
\newcommand{\oX}{\mathbf{X}}

\newcommand{\cD}{\mathcal{D}}
\newcommand{\rD}{\mathcal{D}}

\newcommand{\eins}{{\mathbbm{1}}}
\newcommand{\bas}{\begin{align*}}
\newcommand{\eas}{\end{align*}}

\renewcommand{\L}{\mathrm{L}}

\renewcommand{\H}{\mathrm{H}}

\newcommand{\la}{\langle}
\newcommand{\ra}{\rangle}

\renewcommand{\d}{\mathrm{d}}

\newcommand{\E}{\mathrm{E}}
\newcommand{\e}{\mathfrak{e}}

\definecolor{rot}{rgb}{1.000,0.000,0.000}

\definecolor{blau}{rgb}{0.000,0.000,1.000}

\newcommand{\oA}{\mathbf{A}}
\newcommand{\oI}{\mathbf{I}}

\newcommand{\oT}{\mathbf{T}}

\newcommand{\oB}{\mathbf{B}}
\newcommand{\oC}{\mathbf{C}}
\newcommand{\oD}{\mathbf{D}}
\newcommand{\oL}{\mathbf{L}}
\newcommand{\oM}{\mathbf{M}}

\newcommand{\oQ}{\mathbf{Q}}

\makeatother

\usepackage{babel}
\providecommand{\assumptionname}{Assumption}
\providecommand{\corollaryname}{Corollary}
\providecommand{\lemmaname}{Lemma}
\providecommand{\propositionname}{Proposition}
\providecommand{\remarkname}{Remark}
\providecommand{\theoremname}{Theorem}

\begin{document}
\title{Positivity and polynomial decay of energies for square-field operators\thanks{Research supported by DFG via SFB 1114 (project no.235221301, subproject C05).}}
\author{Artur Stephan\thanks{Weierstraß-Institut für Angewandte Analysis und Stochastik, Mohrenstraße 39, 10117 Berlin, Germany, e-mail: \tt{artur.stephan@wias-berlin.de}} ~and Holger Stephan\thanks{Weierstraß-Institut für Angewandte Analysis und Stochastik, Mohrenstraße 39, 10117 Berlin, Germany, e-mail: \tt{holger.stephan@wias-berlin.de}}}
\date{December 8, 2021}
\maketitle

\lhead{On square-field operators}

\rhead{Artur and Holger Stephan}

\chead{December 8, 2021}


\begin{abstract}
We show that for a general Markov generator the associated square-field (or
carré du champs) operator and all their iterations are positive. The proof is 
based on an interpolation between the operators involving the generator and
their semigroups, and an interplay between positivity and convexity on Banach lattices. Positivity of the square-field operators allows to define a
hierarchy of quadratic and positive energy functionals which decay to zero
along solutions of the corresponding evolution equation. Assuming that the
Markov generator satisfies an operator-theoretic normality condition, the
sequence of energies is log-convex. In particular, this implies polynomial decay in time for the
energy functionals along solutions.  
\end{abstract}



\section{Introduction}

The study of the long time behavior of solutions of differential equations 
\begin{align}\label{e001}
\dot{g}(t)=\oA g(t),~g(0)=g_{0}\in\mathcal{X},
\end{align}	
where $\oA$ is a (in general unbounded) linear operator in a Banach space $\mathcal{X}$, plays a major role in mathematical physics. Usually, the time asymptotic behavior is studied by investigating energy functionals along the solution. Let $\E:\mathcal{X}\to\R$ be such an energy functional. If $\E$ is positive, i.e. $\E(g)\geq0$ for $g\in\rX$, and it decays along the solution, i.e., if the function $t\to\E(g(t))$ has a negative derivative $\frac{\d}{\d t}\E(g(t))\leq0$, then one can at least deduce the existence of a limit $\lim_{t\to\infty}\E(g(t))=\E(g(\infty))\leq\E(g(0))$.

The prototypical example is given by the scalar diffusion equation $\dot{g}=\tfrac 1 2\Delta g$ on $\Omega\subset\R^{d}$ with no-flux boundary condition. We define the quadratic energy $\E_{0}(g)=\int_{\Omega}g^{2}\d x$, and, furthermore, the functionals $\E_{1}(g):=\int_{\Omega}|\nabla g|^{2}\d x$ and $\E_{2}(g):=\int_{\Omega}|\Delta g|^{2}\d x$. Obviously, we have $\E_{i}\geq0$. Then, along (sufficiently smooth) solutions $t\mapsto g(t)$ of $\dot{g}=\tfrac 1 2\Delta g$, we have $$\frac{\d}{\d t}\E_{0}(g(t))=-\E_{1}(g(t)),\quad\frac{\d}{\d t}\E_{1}(g(t))=-\E_{2}(g(t)),$$
showing the decay of $t\mapsto\E_{0}(g(t))$ and $t\mapsto\E_{1}(g(t))$ along
solutions. Often, the functional $\E_{1}$ is called \textit{dissipation rate}, and
$\E_{2}$ would describe the “dissipative loss” of the dissipation. In the
following, we will call an operator $\oA$ or the corresponding equation of
type (\ref{e001}) \textit{dissipative} if the dissipation rate is
positive. The asymptotic decay can be quantified by, for example, proving a
Poincaré-type inequality $\E_{k+1} \geq c~ \E_k$ with some $c>0$. Then exponential decay is obtained for $t\mapsto \E_k(g(t))$ by the classical Gronwall lemma.

Extending dissipativity to other systems usually operators in divergence form  
$\oA = - \frac 1 2\oD^* \oL \oD$ in a suitable Hilbert space are considered. Here  $\oD$ is an operator, containing the 
constant function in the kernel (e.g., the gradient), $\oD^*$ is its adjoint 
with respect to a scalar product $(\cdot,\cdot)$ in a suitable
Hilbert space, $\oL$ is a positive multiplication operator such that $-\oA$ is a symmetric operator that is positive in the form sense.  Defining the energy 
$\E_{0}(g)=(g,g)$, we obtain 
\[
\frac{\d}{\d t}\E_{0}(g(t))=(g,\oA g) = - (g,\oD^* \oL \oD g) = 
- (\oD g,\oL \oD g)=:-\E_1(g)\,.
\]
Since $\oL$ is a positive multiplication operator, we have $\E_1(g) \geq 0$, and, hence,  operators in divergence form are dissipative by construction. 
Moreover, it is easy to check that 
for functionals
\begin{align}\label{e002}
\E_k(g) = (g,(-\oA)^k g)
\end{align}
we have $
\E_k(g) \geq 0$ and $
\frac{\d}{\d t}\E_{k}(g(t)) = - \E_{k+1}(g(t))$. 

As \textit{dissipativity}, another important physical property is \textit{positivity}. A solution $g$, describing for example the concentration or density, should be certainly positive for being physical reasonable\footnote{In theory, there is no reason for a system not to be dissipative.}, and this property of the solution has often to be either proved by hand or assumed. The aim of the paper is to show that dissipativity is the mathematical consequence of the general notion of \textit{positivity} from the theory of
Banach lattices.  

Linear operators $\oA$ in (\ref{e001}) preserve
positivity and mass if and only if they are Markov generators generating a semigroup $(\oT(t))_{t\geq 0}$ of Markov operators. Here, we consider Markov operators on the Banach lattice of continuous functions
$\rC(\rZ)$ on a compact topological space $\rZ$. 
Their adjoints map the space of probability measures to itself, ensuring the system to remain physical reasonable.

To investigate dissipativity of the system, the fundamental object will be the so-called \textit{square-field operators} (or \textit{carré du champs operators}). They have been introduced by Meyer and Bakry \cite{Meye82NPOU, Meye84TRG, Bakr85TRG} and played an important role in the theory of evolution equations and stochastic analysis since then. They are recursively defined by
\begin{align*}
\Gamma_0(f,g) & := f \cdot g\\
\Gamma_{n+1}(f,g) & :=\oA\Gamma_{n}(f,g)-\Gamma_{n}(\oA f,g)-\Gamma_{n}(f,\oA g),
\end{align*}
and measure the difference from $\oA$ being a derivation (i.e. $\oA$ satisfying $\oA (f\cdot g) =\oA f\cdot g+f\cdot \oA g $). In the following, we use the notation $\Gamma_n(g)=\Gamma_n(g,g)$ for evaluating at the diagonal $f=g$. Many interesting features have been investigated for square-field operators in the last decades, connecting analysis, stochastics and geometry. The pioneering Bakry-Émery \cite{BakEme85DH, Ledo00GMDG} condition $\Gamma_2\geq c\, \Gamma_1$ connects geometric properties with analytic functional inequalities. They are of great interest particularly for diffusion Markov generators, where $\oA$ is given by, generally speaking, $\Delta - \nabla V\cdot \nabla $ (we refer to \cite{BaGeLe14AGMDO} and reference therein concerning diffusion Markov generators). In this paper, we do not assume any particular structure of the Markov generator, but aim at showing general inequalities. Inequalities for square-field operators and geometric curvature bounds provide insights also in spectral properties and exponential decay to equilibrium, which has been further investigated and exploited, e.g.  \cite{AMTU01CSIR, MarVil00TEFPE}. In particular, the Bakry-Émery condition ensures that $\Gamma_2\geq 0$, which has geometric  implications for the underlying manifold and was one of the starting point for the research on square-field operators (see \cite{Bakr85TRG}). We also remark that in \cite{Wu00TISDS} different functional inequalities have been derived under the condition that higher iterations, namely $\Gamma_3$, are positive. 

In this paper, our first main result is that for a general Markov generator
$\oA$, all iterated square-field operators evaluated at the diagonal are
positive, i.e. $\Gamma_n(g,g)\geq 0$ whenever they are defined (see Theorem
\ref{thm:PositivityGknandGamman}). To show this, we do not use any restrictive
symmetry assumption on $\oA$ (like detailed balance) or assumptions on the
spectrum. The proof exploits an iterative interpolation between $\Gamma_n$
involving the Markov generator and a version involving only the associated
semigroup of (bounded) Markov operators.  The idea is to start with the
semigroup of Markov operators, derive inequalities for them and transfer them
afterwards to the corresponding Markov generators. The crucial observation
  is the interplay of positivity and
  convexity by a parallelogram identity as an inherent feature of Banach lattices (see Lemma \ref{lem:PositivityGn}).
For Markov operators we may use classical inequalities and circumvent technical difficulties that occur for unbounded operators. However, technical difficulties arise because differentiation does not preserve positivity.

To show the main idea  for $\Gamma_1$ (basically as in \cite{Ledo00GMDG}), we aim at showing that
$\Gamma_1(g,g) =\oA g^2-2 g \cdot\oA g \geq 0$ for all $g\in\cD(\oA)$ with $g^2\in\rD(\oA)$. 
From Jensen's inequality we know that for
all Markov operator $\oM$ and all $g\in\cD(\oA)$ we have $\oM g^2 - (\oM g)^2 \geq 0$. Hence, 
$\oT(t) g^2 - (\oT(t) g)^2 \geq 0$. 
Differentiating in time at $t=0$, we obtain for
$g \in \rD(\oA)$
\[
{\d \over \d t} \left( \oT(t) g^2 - (\oT(t) g)^2 \right) |_{t=0}=
\oT'(t) g^2 - 2 \oT(t) g \cdot  \oT'(t) g |_{t=0} = \oA g^2-2 g \cdot\oA g = \Gamma_1(g,g).
\]
Although, differentiating generally does not preserve inequalities, the reasoning holds true because we have the decomposition
\begin{align}\label{eq:Decomposition}
\oT(t) g^2 - (\oT(t) g)^2 = 
(\oT(t)-\oI) g^2 - g \cdot (\oT(t) -\oI) g  - \oT(t) g \cdot (\oT(t) -\oI) g,  
\end{align}
where $\oI$ is the identity operator. Then we have
\begin{align*}
{\d \over \d t} &\left( \oT(t) g^2 - (\oT(t) g)^2 \right) |_{t=0}
 = 
\lim_{t\to 0} {1 \over t} \left( \oT(t) g^2 - (\oT(t) g)^2 \right) - 
\left( \oT(0) g^2 - (\oT(0) g)^2 \right) \\
& =  
\lim_{t\to 0} {1 \over t}
\left( \oT(t) g^2 - (\oT(t) g)^2 \right) = 
\lim_{t\to 0} {1 \over t}
\left(
(\oT(t)-\oI) g^2 - g \cdot (\oT(t) -\oI) g  - \oT(t) g \cdot (\oT(t) -\oI) g  
\right)\\
& = 
\oA g^2- g \cdot\oA g - g \cdot\oA g  = \Gamma(g,g),
\end{align*}
which proves the desired estimate. 
This idea can be used also for higher iterations $\Gamma_n$, $n\geq 1$, and, moreover, a whole family of new inequalities involving $\oA$ and its semigroup can be obtained as a byproduct.

Positivity of $\Gamma_n$ implies that the corresponding energies\footnote{Note that we call them all \textit{energies} and neglect the physical interpretation.}
\begin{align*}
\E_n(g) := \langle \Gamma_n (g),\mu\rangle,\quad \e_n(t) := \E_n(g(t)) 
\end{align*}
decay along solutions $\dot g=\oA g$ and, moreover, are convex (see Theorem
\ref{thm:PropertiesEnergies}). Here $\mu$ is a stationary measure,
i.e. $\oA^*\mu=0$. This fact is well known for $\E_0$ and also studied for
$\E_1$, which corresponds to the Dirichlet form of $\oA$. In the present
  paper we show that 
all iterations $\e_n$ monotonically decay to zero. This shows in fact, that
dissipativity is not a special physical property but instead a mathematical
consequence of a the natural physical property of positivity.

In the second part of the paper, we quantify the convergence rate of $t\mapsto \e_n(t)$ towards zero. Usually exponential decay is derived assuming a Poincaré-type inequality between, say, $\E_n(g)$ and $\E_{n+1}(g)$. It is clear, that exponential decay can only be expected whenever a spectral gap for the operator $\oA$ is present. In particular, for general Markov generators no exponential decay is expected. To obtain nevertheless quantitative convergence for general Markov generators without assumptions on their spectrum, we adapt the operator-theoretic notion of \textit{normality} for the lifted version of $\oA$ in the natural Hilbert space $\L^2(\mu)$. This enables to prove that the sequence $(\E_n(g))_{n\in\N}$ is log-convex implying many interesting features. First, polynomial decay of the energies $t\mapsto \e_n(t)$ can be derived (Theorem \ref{thm:PolynomialDecayEnergy}). Secondly, log-convexity in time and also an upper bound for powers of energies can be shown (see Theorem \ref{thm:LogConvexityInTime}). We finally note that normality is much general than symmetry, which would translate to detailed balance or reversibility for stochastic processes.   We refer to Section \ref{sec:Discussion-of-normality} where normality is discussed.

\section{Positivity of quadratic operators in $\rC(\rZ)$}

Before defining the square-field operators and proving their positivity, we first briefly present the mathematical setting.

\subsection{Markov semigroups and their generators}

Let us start with a few well known facts from the theory of Markov generators
and their semigroups (see e.g. \cite{Aren86OPSPO,Pazy83SLOA} for more details). 
In what follows, let ${\cal Z}$ be a compact (if necessary, suitably
compactified) and metrizable topological space, i.e. a compact,
first-countable, Hausdorff space.
Let ${\cal C(Z)}$ be the space of continuous real-valued functions on ${\cal Z}$
and 
${\cal C^*(Z)}$ (the dual of ${\cal C(Z)}$) the space of 
Radon measures on Borel sets generated by the open sets of $\cal Z$. The space ${\cal C(Z)}$ is a Banach algebra by the pointwise multiplication (denoted by
$\cdot$); the constant
function ${\mathbbm{1}}(z) \equiv 1$ is denoted by ${\mathbbm{1}} \in {\cal
  C(Z)}$. The dual pairing is denoted by $\langle g,p \rangle = \int_\rZ g(z)
p(\d z)$ with 
$g \in {\cal C(Z)}$ and $p \in {\cal C^*(Z)}$.

The spaces ${\cal C(Z)}$ and ${\cal C^*(Z)}$ are
Banach lattices with the order relations
${\cal C(Z)} \ni g \geq 0$ if and only if for all $z\in\cal Z$ we have $g(z) \geq 0$, and 
${\cal C^*(Z)} \ni p \geq 0$ if and only if for all Borel sets $B\subset \rZ$ we have $ p(B) \geq 0$. 
The order relation in ${\cal C^*(Z)}$ as a space of measures coincides with
the order relation 
in dual spaces, i.e. $p \geq 0$ if and only if  $\langle g,p \rangle \geq 0$ for all $0 \leq g \in {\cal C}({\cal Z})$. In the following, elements $g$ or $p$ with $g \geq 0$ and $p \geq 0$ are called \textit{positive}\footnote{Throughout the paper, \textit{positive} is meant to be non-negative. Moreover, \textit{positivity} of functions is defined pointwise and not ``almost everywhere''.}. The convex subset
${\cal P}({\cal Z}) = \big\{ p \in {\cal C}^*({\cal Z})~\big|~ p \geq
0,~\langle \eins,p \rangle =\|p\| = 1\big\}~$ 
is the set of probability measures describing the statistical states of the system. 

As usual, by ${\cal L}({\cal C(\cal Z)})$ and ${\cal L}({\cal C(\cal Z)}^*)$ we denote the spaces of linear
bounded operators. 
A linear operator $\mathbf{T} \in {\cal L}({\cal C})$ on a Banach lattice is
called \textit{positive} (written $\mathbf{T} \geq 0$) if it conserves positivity,
i.e., $g \geq 0$ implies $\oT g \geq 0$. 
A linear, bounded operator $\mathbf{M} \in {\cal L}({\cal C(\cal Z)})$ with
$\mathbf{M} {\mathbbm{1}} = {\mathbbm{1}}$ and $\mathbf{M} \geq 0$ is called \textit{Markov operator}. The set of Markov operators
${\cal M} = \big.\big\{ \mathbf{M} \in {\cal L}({\cal C}) ~\big|~\mathbf{M}
\geq 0,~\mathbf{M} {\mathbbm{1}} = {\mathbbm{1}} \big\}$ 
is convex and constitute a (noncommutative) semigroup with unit $\oI$, where $\oI$ is the identity operator on $\cal C(\cal Z)$. A 
Markov operator is contractive, because we have $\|\oM\|=1$. Important for us will be 
Jensen's inequality, which says that for all convex function $\Phi:\R \to [-\infty,\infty]$, all Markov operators $\oM \in \rM$ and all $g \in \rC(\rZ)$  we have
\begin{equation}\label{e131}
\oM \Phi(g) \geq \Phi(\oM g), 
\end{equation}
where $\Phi(g)$ is defined pointwise, i.e. $\Phi(g)(z)=\Phi(g(z))$. The simple proof for that, based on the fact that Markov operators provide convex
combinations, is contained in \cite{Step05LFPLEP}.

Adjoints of Markov operators $\oM^*$ are also positive and 
contractive with 
$\| \mathbf{M}^* \| = 1$.
Given an operator $\oT^* \in {\cal L}({\cal C(\cal Z)}^*)$ we have  $\oT^* {\cal P(\cal Z)} \subset  {\cal P(\cal Z)} $ if and only if $\oT \in \rM$.
In this sense, we
can say that Markov operators and only these provide physical reasonable state
changes. For all $\oM \in {\cal M}$ there is at least one invariant
measure $\mu \in \rP$ with $\mathbf{M}^* \mu = \mu$ by the Frobenius--Perron--Krein--Rutman
Theorem.  

A Markov semigroup $(\mathbf{T}(t))_{t\geq 0}$ is strongly-continuous semigroup of Markov operators, i.e. 
$\mathbf{T}(0)=\mathbf{I}$, $\mathbf{T}(t_1+t_2) = \mathbf{T}(t_2)
\mathbf{T}(t_1)$ for all $t_1,t_2\geq 0$, and 
$\mathbf{T}(t)g$ converges strongly to $g$ as $t \to 0$.
A Markov semigroup is contractive, since for all $t\geq 0$ we have $\|\oT(t)\| = 1$.
Moreover, a Markov semigroup is a commuting family of Markov operators, and therefore,
there exists an invariant measure $\mu \in {\cal P}({\cal Z})$  not depending on $t$ with
$\mathbf{T}^*(t) \mu = \mu$ by the Markov--Kakutani Theorem (see e.g. \cite{DunSch59LO1}). For a given strongly-continuous
semigroup $({\bf T}(t))_{t\geq0}$ the Markov generator is denoted
by $({\bf A},\mathcal{D}(\oA))$, where 
\[
\mathcal{D}({\bf A}):=\left\{
  g\in\mathcal{C}(\mathcal{Z}):\lim_{t\to0}\frac{1}{t}\left({\bf
      T}(t)g-g\right)\ \mathrm{exists}\right\} ,\quad {\bf
  A}g:=\lim_{t\to0}\frac{1}{t}\left({\bf T}(t)g-g\right),~g \in
\mathcal{D}({\bf A}) 
\,.
\]
Denoting $\mathcal{D}(\oA^{n})$ the domain of $\oA^{n}$ and setting 
$\bigcap_{n\in\N}\mathcal{D}(\oA^{n})=:\mathcal{D}(\oA^{\infty})$, we have that $\mathcal{D}(\oA^{\infty})$ is a core of $(\oA,\mathcal{D}(\oA))$ and a dense subset of $\mathcal{C}(\mathcal{Z})$ (see, e.g., \cite{Pazy83SLOA}). In particular, in the following we will sometimes omit denoting the
domain of definition for operators $\oA^n$, 
because obviously all identities only hold, if both sides of the relation
are well-defined. Since
$\mathcal{D}(\oA^{\infty})$ is dense, we may always
consider $f,g \in \mathcal{D}(\oA^{\infty})$.

Given a differential equation 
\begin{align}\label{eq:MarkovProcess}
\dot{g} = \oA g,\quad g(0)=g_0 \in
\mathcal{D}(\oA)\subset \rC(\rZ),
\end{align}
where $\oA$ is a Markov generator with semigroup $(\mathbf{T}(t))_{t\geq 0}$, then the solution $t\mapsto g(t)$
is given by  
 $g(t) = \oT(t) g_0$.
If $\mu$ is an invariant measure of a $(\mathbf{T}(t))_{t\geq 0}$, then
$\mu \in \mathcal{D}(\oAs )$ and $\oAs  \mu = 0$ and vice versa: If a
measure $\mu \in \mathcal{D}(\oAs )$ satisfy $\oAs  \mu = 0$
then we have $\oT^*(t) \mu = \mu$. 


In mathematical physics and stochastics, equations of the type (\ref{eq:MarkovProcess}) describe the time evolution of
observations and is called \textit{(Chapman-Kolmogorov) backward equation}.  From physical perspective, the so-called
\textit{(Chapman-Kolmogorov) forward equation} for the evolution of a probability measures $t\mapsto p(t)\in\rP(\rZ)$ is interesting. It can be
formally expressed as $\dot{p}(t) = \oA^* p(t)$, $p(0)=p_0$, which is defined in the weak-* sense by
\begin{align}\label{e011}
{\d \over \d t} \la g,p(t) \ra = \la \oA g,p(t) \ra, \quad g \in \rD(\oA), \quad p(0)=p_0.
\end{align}
Although in general $\oA^*$ is not a generator (it does not need to be densely defined), the solution of (\ref{e011}) can be found by solving the equation
(\ref{eq:MarkovProcess}), finding $\oT(t)$ and setting $p(t) = \oT^*(t) p_0$.

Assuming that $p_0$ has a density $h_0$ with respect to $\mu$, $p(t)$ can
also be calculated as $p(t) = h(t)\mu$, where $h(t)$ is the solution to an
equation of type (\ref{e001}) but with a different Markov operator as
$\oA$. In this form, the equation  is commonly used as Fokker--Planck
equation \cite{AMTU01CSIR}, Levy--Fokker--Planck equation \cite{GenImb08},
master equation and others. Due to their universality, we restrict
ourselves in this paper to equation (\ref{eq:MarkovProcess}). 

\subsection{Square-field operators and their interpolation}\label{subsec:SFO}

For a given Markov generator $(\oA,\rD(\oA))$, we define the so-called 
\textit{square-field operators} by
\begin{align*}
\Gamma_{0}(f,g) & =f\cdot g\\
\Gamma_{n+1}(f,g) & =\oA\Gamma_{n}(f,g)-\Gamma_{n}(\oA f,g)-\Gamma_{n}(f,\oA g).
\end{align*}
Of course, to ensure that $\Gamma_n$ is well-defined, we have to restrict to a subalgebra and to elements of $\rD(\oA^n)$. To do so, we define $\rD:=\rD(\oA^\infty)\cap\big(\rD(\oA^\infty)\cdot \rD(\oA^\infty)\big)$ the largest subalgebra\footnote{To decide whether a subalgebra $\rD\subset\rC(\rZ)$ is dense, it suffices to show that it separates points by the Stone-Weierstraß-Theorem, see e.g. \cite{Sema71BSCF}.} contained in $\rD(\oA^\infty)$. Here, in Section \ref{subsec:SFO} and Section \ref{subsec:PositivityForGammaAndG} we always assume that $f,g\in\rD$. Note that for the associated energies in Section \ref{subsec:Energies}, no subalgebra is needed and it suffices to consider functions from $\rD(\oA^\infty)$.

The first elements are given by
\begin{align*}
\Gamma_{1}(f,g) & =\oA(f\cdot g)-f\cdot\oA g-g\cdot\oA f\\
\Gamma_{2}(f,g) & =\oA^{2}(f\cdot g)-2\oA(f\cdot\oA g)-2\oA(g\cdot\oA f)+g\cdot\oA^{2}f+2\oA g\cdot\oA f+f\cdot\oA^{2}g.
\end{align*}

Moreover, we define the following square-field type operators which contain bounded Markov operators. For this, consider a sequence of Markov operators $\left(\oM_{n}\right)_{n\in\mathbb{N}}$. We define a sequence $G_{n}=G_{n}(\oM_{n},\cdots,\oM_{1},\cdot,\cdot)$
of quadratic operators on $\mathcal{C}(\mathcal{Z})$ by the recursion formula
\begin{align}
G_{0}(f,g) & =f\cdot g\nonumber \\
G_{n+1}(\oM_{n+1},...,\oM_{1},f,g) & =\oM_{n+1}G_{n}(\oM_{n},...,\oM_{1},f,g)-G_{n}(\oM_{n},...,\oM_{1},\oM_{n+1}f,\oM_{n+1}g)\,. 
\label{eq:Definition=00005Cm_n-1-1}   
\end{align}
The first terms of $G_{n}$ are given by
\begin{align*}
G_{1}(\oM_{1},f,g) & =\oM_{1}(f\cdot g)-\oM_{1}f\cdot\oM_{1}g\\
G_{2}(\oM_{2},\oM_{1},f,g) & =\oM_{2}\oM_{1}(f\cdot
g)-\oM_{2}(\oM_{1}f\cdot\oM_{1}g)-\oM_{1}(\oM_{2}f\cdot\oM_{2}g)+ (\oM_{1}\oM_{2}f)\cdot(\oM_{1}\oM_{2}g).
\end{align*}
Note that $G_n$ has $n$ Markov operators in their argument. Since $G_{0}$ is symmetric and linear in both arguments separately
the same holds for all $G_{n}$ by definition. Moreover, if
$f,g\in\rC(\rZ)$, then also $G_{n}(\oM_{n},...,\oM_{1},\dots,f,g)\in\rC(\rZ)$.

As we will see the operators $\Gamma_{n}$ can be obtained from $G_{n}$
by subsequently substituting $\oM_{k}$ by the semigroup ${\bf T}(t)$
of the Markov generator $\oA$ and taking the limit $t\to0$. For this, we introduce operators
$G_{n}^{k}$, which interpolate between $G_{n}$ and $\Gamma_{n}$,
and are defined by
\begin{align*}
G_{n}^{0}(\oM_{n},\dots,\oM_{1},f,g) & =G_{n}(\oM_{n},\dots,\oM_{1},f,g),\\
G_{n}^{k}(\oM_{n-k},\dots,\oM_{1},f,g) & =\oA G_{n-1}^{k-1}(\oM_{n-k},\dots,\oM_{1},f,g)\\
 & \quad-G_{n-1}^{k-1}(\oM_{n-k},\dots,\oM_{1},\oA f,g)-G_{n-1}^{k-1}(\oM_{n-k},\dots,\oM_{1},f,\oA g)\,.
\end{align*}
By definition, $G_{n}^{k}(f,g)$ is only defined if $G_{n-1}^{k-1}(f,g)\in\mathcal{D}(\oA)$
and $G_{n-1}^{k-1}(\oM_{n-k},\dots,\oM_{1},\oA f,g)$ as well as $G_{n-1}^{k-1}(\oM_{n-k},\dots,\oM_{1},f,\oA g)$
are well-defined. We note that $G_n^k$ has $n{-}k$ Markov operators in their arguments and all $G^k_n$ are symmetric.

Next, we are going to show how to obtain $\Gamma_n$ from $G_n$ to carry over
$G_n \geq 0$ to $\Gamma_n \geq 0$. First, in Lemma \ref{lem:PropertiesGkn}, we show that $\Gamma_n=G^n_n$. Secondly, in Proposition \ref{prop:ConnectionGknToGk+1n} we show how to obtain $G^k_{n+1}$ from $G^k_n$. The connection between the operators can be summarized in the following diagram:
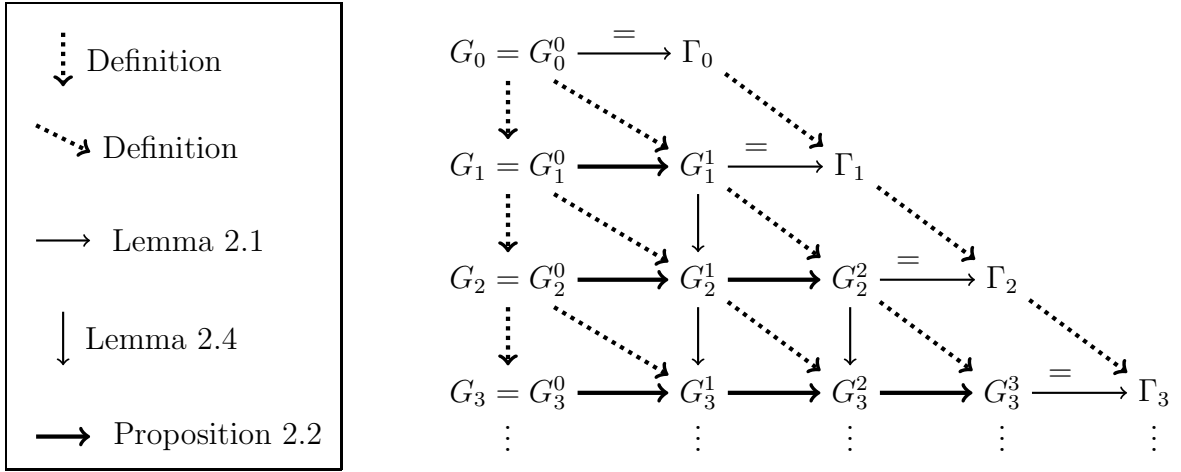
\begin{figure}[h]
\hspace{1cm}
\boxed{
\begin{tikzpicture}[scale=1.0]
\node (n0) at (0,5.5) {};
\node (n1) at (0,4.5) {};

\node (n2) at (-0.5,4.2) {};
\node (n3) at (0.5,3.7) {};

\node (n4) at (-0.5,2.6) {};
\node (n5) at (0.5,2.6) {};

\node (n6) at (0,1.8) {};
\node (n7) at (0,0.8) {};

\node (n8) at (-0.5,0) {};
\node (n9) at (0.5,0) {};
\draw[dotted, ultra thick, ->] (n0)-- node[pos=0.5,right,color=black]{~Definition}(n1) ; 
\draw[dotted, ultra thick, ->] (n2)-- node[pos=0.8,right,color=black]{~Definition}(n3) ; 
\draw[thick, ->] (n4)-- node[pos=1.0,right,color=black]{~Lemma \ref{lem:PropertiesGkn}}(n5) ; 
\draw[thick, ->] (n6)-- node[pos=.5,right,color=black]{~Lemma \ref{lem:DecompositionGkn}}(n7) ; 
\draw[ultra thick, ->] (n8)-- node[pos=1.0,right,color=black]{~Proposition \ref{prop:ConnectionGknToGk+1n}}(n9) ; 
\end{tikzpicture}
}
\hspace{1cm}
\begin{tikzpicture}
\node (n0) at (0.5,4.5) {$G_0 = G_0^0$};
\node (n1) at (0.5,3) {$G_1 = G_1^0$};
\node (n2) at (0.5,1.5) {$G_2 = G_2^0$};
\node (n3) at (0.5,0) {$G_3 = G_3^0$};
\node (n5) at (3,3) {$G^1_1$};
\node (n6) at (3,1.5) {$G^1_2$};
\node (n7) at (3,0) {$G^1_3$};
\node (n8) at (5,1.5) {$G^2_2$};
\node (n9) at (5,0) {$G^2_3$};
\node (n10) at (7,0) {$G^3_3$};
\node (n11) at (3,4.5) {$\Gamma_0$};
\node (n12) at (5,3) {$\Gamma_1$};
\node (n13) at (7,1.5) {$\Gamma_2$};
\node (n14) at (9,0) {$\Gamma_3$};
\node (n15) at (0.5,-0.5) {$\vdots$};
\node (n15) at (3,-0.5) {$\vdots$};
\node (n16) at (5,-0.5) {$\vdots$};
\node (n17) at (7,-0.5) {$\vdots$};
\node (n18) at (9,-0.5) {$\vdots$};
\draw[thick, ->] (n0)-- node[pos=0.5,above,color=black]{=}(n11) ; 
\draw[thick, ->] (n5)-- node[pos=0.3,above,color=black]{=}(n12) ;
\draw[thick, ->] (n8)-- node[pos=0.3,above,color=black]{=}(n13) ;
\draw[thick, ->] (n10)-- node[pos=0.3,above,color=black]{=}(n14) ;
\draw[ultra thick, ->, dotted] (n0)-- node[pos=0.5,right,color=black]{}(n5) ;
\draw[ultra thick, ->, dotted] (n1)-- node[pos=0.5,right,color=black]{}(n6) ; 
\draw[ultra thick, ->, dotted] (n2)-- node[pos=0.5,right,color=black]{}(n7) ; 
\draw[ultra thick, ->, dotted] (n5)-- node[pos=0.5,right,color=black]{}(n8) ; 
\draw[ultra thick, ->, dotted] (n6)-- node[pos=0.5,right,color=black]{}(n9) ; 
\draw[ultra thick, ->, dotted] (n8)-- node[pos=0.5,right,color=black]{}(n10) ; 
\draw[ultra thick, ->, dotted] (n11)-- node[pos=0.5,right,color=black]{}(n12) ;
\draw[ultra thick, ->, dotted] (n12)-- node[pos=0.5,right,color=black]{}(n13) ; 
\draw[ultra thick, ->, dotted] (n13)-- node[pos=0.5,right,color=black]{}(n14) ; 
\draw[ultra thick, ->, dotted] (n0)-- node[pos=0.5,right,color=black]{}(n1) ;
\draw[ultra thick, ->, dotted] (n1)-- node[pos=0.5,right,color=black]{}(n2) ; 
\draw[ultra thick, ->, dotted] (n2)-- node[pos=0.5,right,color=black]{}(n3) ; 
\draw[thick, ->] (n5)-- node[pos=0.5,right,color=black]{}(n6) ; 
\draw[thick, ->] (n6)-- node[pos=0.5,right,color=black]{}(n7) ; 
\draw[thick, ->] (n8)-- node[pos=0.5,right,color=black]{}(n9) ; 
\draw[ultra thick, ->] (n1)-- node[pos=0.5,above,color=black]{}(n5) ; 
\draw[ultra thick, ->] (n2)-- node[pos=0.5,above,color=black]{}(n6) ; 
\draw[ultra thick, ->] (n3)-- node[pos=0.5,above,color=black]{}(n7) ; 
\draw[ultra thick, ->] (n6)-- node[pos=0.5,above,color=black]{}(n8) ; 
\draw[ultra thick, ->] (n7)-- node[pos=0.5,above,color=black]{}(n9) ; 
\draw[ultra thick, ->] (n9)-- node[pos=0.5,above,color=black]{}(n10) ; 
\end{tikzpicture}
\end{figure}

\begin{lem}
\label{lem:PropertiesGkn} Let $\left(\oM_{n}\right)_{n\in\mathbb{N}}$ be a sequence of Markov operators. The operators $G_{n}^{k}$ have the following
properties:
\begin{enumerate}
\item $G_n^k$ is continuous, i.e., if $f_{t}\to f$, $g_{t}\to g$ in $\mathcal{C}(\mathcal{Z})$ as $t\to 0$,
then $G_{n}^{k}(\oM_{n-k},\dots,\oM_{1},f_{t},g_{t}) \to G_{n}^{k}(\oM_{n-k},\dots,\oM_{1},f,g)$ whenever $G_{n}^{k}(\oM_{n-k},\dots,\oM_{1},f,g)$
is defined.
\item For all $n\in\N$, we have $G_{n}^{n}(f,g)=\Gamma_{n}(f,g)$.
\end{enumerate}
\end{lem}

\begin{proof}
To simplify notation, we just write $G_{n}^{k}=G_{n}^{k}(\oM_{n-k},\dots,\oM_{1},\cdot,\cdot)$ because all $\oM_k$ are fixed.
For the first claim, we observe that $f_{t}\to f$ and $g_{t}\to g$
implies $G_{n}^{0}(f_{t},g_{t})\to G_{n}^{0}(f,g)$ since all $\oM_{j}$
are bounded and $\mathcal{C}(\mathcal{Z})$ is a Banach algebra. By
induction, we obtain that also $G_{n}^{k}(f_{t},g_{t})\to G_{n}^{k}(f,g)$
because $\oA$ is a generator and hence a closed operator.

The proof of the second claim is done by induction. By definition,
the claim holds for $n=0$. The recursion formula for $G_{n}^{n}$
is given by
\[
G_{n}^{n}(f,g)=\oA G_{n-1}^{n-1}(f,g)-G_{n-1}^{n-1}(\oA f,g)-G_{n-1}^{n-1}(f,\oA g)\,.
\]
This is the same for $\Gamma_{n}$.
\end{proof}
The following proposition relates $G_{n}^{k}$ with $G_{n}^{k+1}$
by replacing $\oM_{n-k}$ with ${\bf T}(t)$ and taking the rescaled
limit $t\to0$.
\begin{prop}
\label{prop:ConnectionGknToGk+1n}Let a sequence of Markov operators
$\big({\bf M}_{n}\big)_{n\in\N}$ be given such that ${\bf M}_{n}$ commutes with the
semigroup ${\bf T}(t)$. Then, for all $n\in\N$ and $k\in\left\{ 0,\dots,n-1\right\} $
we have 
\[
\lim_{t\to0}\ \frac{1}{t}\ G_{n}^{k}({\bf T}(t),\oM_{n-k-1},\dots,\oM_{1},f,g)=G_{n}^{k+1}(\oM_{n-k-1},\dots,\oM_{1},f,g)\,.
\]
\end{prop}

Before proving the proposition, we immediately observe the following. Replacing subsequently all Markov operators ${\bf M}_{k}$ by the
semigroup ${\bf T}(t_{k})$ we derive the following connection
between $G_{n}$ and $\Gamma_{n}$.
\begin{cor}
We have 
\[
\lim_{t_{1}\to0}\ \frac{1}{t_{1}}
\left(\lim_{t_{2}\to0}\
  \frac{1}{t_{2}}\left(\dots
\lim_{t_{n}\to0}\
\frac{1}{t_{n}}\
      G_{n}({\bf T}(t_{n}),\dots,{\bf
        T}(t_{1}),f,g)
\right)
\right)
=\Gamma_{n}(f,g)\,. 
\]
\end{cor}

\begin{proof}
This follows directly from the definition $G_{n}^{0}=G_{n}$, Proposition
\ref{prop:ConnectionGknToGk+1n} and the relation $G_{n}^{n}=\Gamma_{n}$
from Lemma \ref{lem:PropertiesGkn}.
\end{proof}
To prove Proposition \ref{prop:ConnectionGknToGk+1n}, we need the
following lemma, which can be understood as the generalization of \eqref{eq:Decomposition} from the Introduction.
\begin{lem}
\label{lem:DecompositionGkn}
Let a sequence of Markov operators
$({\bf M}_{n})_{n\in\N}$ be given such that ${\bf M}_{n-k}$ commutes with the
semigroup ${\bf T}(t)$.  For all $n\in\N$ and $k\in\left\{ 0,\dots,n{-}1\right\} $
we have 
\begin{align*}
G_{n}^{k}(\oM_{n-k},\oM_{n-k-1},\dots,\oM_{1},f,g)=&\oM_{n-k} G_{n-1}^{k}(\oM_{n-k-1},\dots,\oM_{1},f,g)\\
&\quad -G_{n-1}^{k}(\oM_{n-k-1},\dots,\oM_{1},{\bf M}_{n-k}f,{\bf M}_{n-k}g).
\end{align*}
In particular, we have that
\begin{align}\label{eq:Identity}
G_{n}^{k}(\oM_{n-k},\dots,f,g)&=\left(\oM_{n-k}-{\bf
    I}\right)G_{n-1}^{k}(\dots,f,g)-\\
&- G_{n-1}^{k}(\dots,{\bf M}_{n-k}f,\left({\bf M}_{n-k}-{\bf
    I}\right)g)-G_{n-1}^{k}(\dots,\left({\bf M}_{n-k}-{\bf I}\right)f,g)\,\notag. 
\end{align}
\end{lem}

\begin{proof}
To simplify notation, we again just write $G_{n}^{k}(\oM,f,g)=G_{n}^{k}(\oM,\oM_{n-k-1},\dots,\oM_{1},f,g)$ and $G_{n-1}^{k}(f,g)=G_{n-1}^{k}(\oM_{n-k-1},\dots,\oM_{1},f,g)$.
We are going to show that $G_{n}^{k}(\oM,f,g) = \oM G_{n-1}^{k}(f,g)-G_{n-1}^{k}({\bf M}f,{\bf M}g)$;
the second formula follows directly by linearity of $G_{n}^{k}$.

For $k=0$ and all $n\in\N$ this follows from the recursion formula
of $G_{n}^{0}=G_{n}$. Assume that it holds for a fixed $k\geq0$
and all $n\in\N$. We want to show that also 
\[
G_{n+1}^{k+1}(\oM,f,g)=\oM G_{n}^{k+1}(f,g)-G_{n}^{k+1}({\bf M}f,{\bf M}g)\,.
\]
By definition, the left-hand side is -- assuming that $\oM$ and $\oA$
commute -- given by
\begin{align*}
G_{n+1}^{k+1}(\oM,f,g) & =\oA G_{n}^{k}(\oM,f,g)-G_{n}^{k}(\oM,\oA f,g)-G_{n}^{k}(\oM,f,\oA g)\\
 & =\oA\left(\oM G_{n-1}^{k}(f,g)-G_{n-1}^{k}({\bf M}f,{\bf M}g)\right)-\oM G_{n-1}^{k}(\oA f,g)-G_{n-1}^{k}({\bf M}\oA f,{\bf M}g)\\
 & \quad-\oM G_{n-1}^{k}(f,\oA g)-G_{n-1}^{k}({\bf M}f,{\bf M}\oA g)\\
 & =\oM\left\{ \oA G_{n-1}^{k}(f,g)-G_{n-1}^{k}(\oA f,g)-G_{n-1}^{k}(f,\oA g)\right\} \\
 & \quad-\oA G_{n-1}^{k}({\bf M}f,{\bf M}g)+G_{n-1}^{k}(\oA{\bf M}f,{\bf M}g)+G_{n-1}^{k}({\bf M}f,\oA{\bf M}g)\\
 & =\oM G_{n}^{k+1}(f,g)-G_{n}^{k+1}(\oM f,\oM g)\,,
\end{align*}
which is the desired formula.
\end{proof}
Using that lemma, we are able to prove Proposition
\ref{prop:ConnectionGknToGk+1n}. 
\begin{proof}[Proof of Proposition \ref{prop:ConnectionGknToGk+1n}]
Again, we just write $G_{n}^{k}(\oT(t),f,g)=G_{n}^{k}(\oT(t),\oM_{n-k-1},\dots,\oM_{1},f,g)$ and $G_{n-1}^{k}(f,g)=G_{n-1}^{k}(\oM_{n-k-1},\dots,\oM_{1},f,g)$. With \eqref{eq:Identity} from Lemma \ref{lem:DecompositionGkn} and using the linearity of
$G_{n}^{k}$, we have 
\begin{align*}
\frac{1}{t}G_{n}^{k}({\bf T}(t),f,g) & =\frac{1}{t}\left({\bf T}(t){-}{\bf I}\right)G_{n-1}^{k}(f,g)-G_{n-1}^{k}\big({\bf T}(t)f,\frac{1}{t}\left({\bf T}(t){-}{\bf I}\right)g\big) - G_{n-1}^{k}\big(\frac{1}{t}\left({\bf T}(t){-}{\bf I}\right)f,g\big).
\end{align*}
In the limit $t\to0$, the first term converges to ${\bf A}G_{n}^{k}(f,g)$,
if $G_{n}^{k}(f,g)\in\mathcal{D}(\oA)$. Moreover, we have 
\begin{align*}
G_{n-1}^{k}({\bf T}(t)f,\frac{1}{t}\left({\bf T}(t)-{\bf I}\right)g) & \to G_{n-1}^{k}(f,\oA g), \quad G_{n-1}^{k}(\frac{1}{t}\left({\bf T}(t)-{\bf I}\right)f,g) \to G_{n-1}^{k}(\oA f,g),
\end{align*}
as $t\to0$ by continuity (see Lemma \ref{lem:PropertiesGkn}). Hence,
we obtain
\[
\lim_{t\to0}\ \frac{1}{t}\ G_{n}^{k}({\bf T}(t),f,g)=\oA G_{n-1}^{k}(f,g)-G_{n-1}^{k}(\oA f,g)-G_{n-1}^{k}(f,\oA g),
\]
which is equal to $G_{n}^{k+1}(f,g)$ by definition.
\end{proof}
\begin{rem}
We remark that similarly, $\Gamma_{n}$ can be obtained from $G_{n}$
by subsequently inserting $\oM_{k}={\bf T}(t)$ and then differentiating
with respect to $t$ at $t=0$. Indeed, one can show that
\begin{equation*}
\Gamma_{n}(f,g)=\frac{\d}{\d t}\left.\left.\cdots\left\{ \frac{\d}{\d t}\left.\left.\left\{ \frac{\d}{\d t}\left.\left.G_{n}(\oM_{n},\dots,\oM_{1},f,g)\right|_{\oM_{1}={\bf T}(t)}\right|_{_{t=0}}\right\} \right|_{\oM_{2}={\bf T}(t)}\right|_{t=0}\right\} \cdots\right|_{\oM_{n}={\bf T}(t)}\right|_{t=0}\,.
\end{equation*}
Similar formulas also hold true for $G_{n}^{k}$, but
can not directly by used to derive positivity. Moreover, it is possible to derive the following product rule for $\Gamma_n$
\[
\left. {\d \over \d t} \Gamma_n(\oT(t) f,\oT(t) g)\right|_{t=0} =
 \oA \Gamma_n(f,g) - \Gamma_{n+1}(f,g).
\]
\end{rem}

\subsection{Positivity for $\Gamma_n$ and $G_n^k$}\label{subsec:PositivityForGammaAndG}

So far, we introduced $G^k_n$ as an iterative approximation of $\Gamma_n$. Next, we want to show that indeed $G^k_n$ and $\Gamma_n$ are positive, when evaluated at the diagonal $f=g$. For this, with a slight abuse of notation, we introduce 
\begin{align*}
\Gamma_n(g) &:=\Gamma_n(g,g)\\
G_{n}(\oM_{n},\oM_{n-1},\dots,\oM_{1},g) &:= G_{n}(\oM_{n},\oM_{n-1},\dots,\oM_{1},g,g)\\
G_{n}^{k}(\oM_{n-k},\oM_{n-k-1},\dots,\oM_{1},g) &:= G_{n}^{k}(\oM_{n-k},\oM_{n-k-1},\dots,\oM_{1},g,g).
\end{align*}
There will be now confusion with the notation because in this section, we only consider the operators evaluated at the diagonal. The first operators are given by
\begin{align*}
G_{0}(g) & = \Gamma_0(g) = g^{2}\\
G_{1}(\oM,g) & =\oM g^{2}-(\oM g)^{2}\\
G_{2}(\oM_{2},\oM_{1},g) & =\oM_{2}\oM_{1}g^{2}-\oM_{2}(\oM_{1}g)^{2}-\oM_{1}(\oM_{2}g)^{2}+(\oM_{1}\oM_{2}g)^{2}.
\\
\Gamma_{1}(g) & =\oA(g^2)-2g\cdot\oA g\\
\Gamma_{2}(g) & =\oA^{2}(g^2)-4\oA(g\cdot\oA g)+2g\cdot\oA^{2}g+2\oA g\cdot\oA g.
\end{align*}

The next
lemma shows the estimate for $G_{n}=G_{n}^{0}$ which is based on the interplay between convexity and positivity. With the help of
Proposition \ref{prop:ConnectionGknToGk+1n} we are then able to transfer
the estimates to $G_{n}^{k}$ and to $\Gamma_{n}$.
\begin{lem}
\label{lem:PositivityGn}Let a sequence of Markov operators $\left(\oM_{n}\right)_{n\in\mathbb{N}}$
be given. With the above notation, we have:
\begin{enumerate}
\item For all $n\in\mathbb{N}$, the function $g\mapsto
  G_{n}(\oM_{n},\oM_{n-1},\dots,\oM_{1},g)$ 
satisfies the parallelogram identity, i.e. for all
$f,g\in\mathcal{C}(\mathcal{Z})$ 
we have
\begin{align*}
G_{n}&(\oM_{n},\oM_{n-1},\dots,\oM_{1},g)+G_{n}(\oM_{n},\oM_{n-1},\dots,\oM_{1},f) = \\
&\quad 2G_{n}\left(\oM_{n},\oM_{n-1},\dots,\oM_{1},\frac{f+g}{2}\right)+2G_{n}\left(\oM_{n},\oM_{n-1},\dots,\oM_{1},\frac{f-g}{2}\right)\,.
\end{align*}

\item The function $g\mapsto G_{n}(\oM_{n},\oM_{n-1},\dots,\oM_{1},g)$ is convex and non-negative.
\end{enumerate}
\end{lem}

\begin{proof}
We prove the claims in several steps, mainly using induction over
$n\in\mathbb{N}$. Again, we simply write $G_{n}(g)=G_{n}(\oM_{n},\oM_{n-1},\dots,\oM_{1},g) $.

For the first part, we observe that the parallelogram identity holds for $n=0$, since
$G_{0}(g)=g^{2}$ and we have 
\[
f^{2}+g^{2}=2\left(\frac{f+g}{2}\right)^{2}+2\left(\frac{f-g}{2}\right)^{2}\,.
\]
For the step from $n$ to $n+1$, we add the equations
\begin{align*}
G_{n+1}(f) & =\oM_{n+1}G_{n}(f)-G_{n}(\oM_{n+1}f)\\
G_{n+1}(g) & =\oM_{n+1}G_{n}(g)-G_{n}(\oM_{n+1}g),
\end{align*}
and obtain, 
\begin{align*}
 & G_{n+1}(f)+G_{n+1}(g)=\\
 & =\oM_{n+1}G_{n}(f)+\oM_{n+1}G_{n}(g)-\left\{ G_{n}(\oM_{n+1}f)+G_{n}(\oM_{n+1}g)\right\} \\
 & =2\oM_{n+1}\left\{ G_{n}\left(\frac{f+g}{2}\right)+G_{n}\left(\frac{f-g}{2}\right)\right\} -2\left\{ G_{n}\left(\oM_{n+1}\frac{f+g}{2}\right)+G_{n}\left(\oM_{n+1}\frac{f-g}{2}\right)\right\} \\
 & =2\oM_{n+1}G_{n}\left(\frac{f+g}{2}\right)-2G_{n}\left(\oM_{n+1}\frac{f+g}{2}\right)+2\oM_{n+1}G_{n}\left(\frac{f-g}{2}\right)-2G_{n}\left(\oM_{n+1}\frac{f-g}{2}\right)\\
 & =2G_{n+1}\left(\frac{f+g}{2}\right)+2G_{n+1}\left(\frac{f-g}{2}\right)\,,
\end{align*}
where we have used the linearity $\frac{1}{2}\left(\oM_{n+1}f\pm\oM_{n+1}g\right)=\oM_{n+1}\left(\frac{f\pm g}{2}\right)$.
This proves the claim.

The proof of the second part is done by induction in two steps. First
we show that convexity of $G_{n}$ implies that $G_{n+1}\geq0$. Secondly,
we show that if $G_{n}\geq0$ then $G_{n}$ is convex, by using the
parallelogram identity. By induction this would conclude the proof.

Clearly, we have that $G_{0}(g)=g^{2}$ is positive and convex. Assuming that $g\mapsto G_{n}(g)$ is convex, we have for all $\oM_{n+1}$ by Jensen's inequality (\ref{e131}) that
\[
\forall g\in\mathcal{C}(\mathcal{Z}):\ \oM_{n+1}G_{n}(g)\geq G_{n}(\oM_{n+1}g),
\]
which proves the positivity of $G_{n+1}$ because
\[
G_{n+1}(g) =\oM_{n+1}G_{n}(g)-G_{n}(\oM_{n+1}g) \geq 0.
\]
Next, we show that positivity implies convexity. By the parallelogram
identity, we have for all $f,g\in\mathcal{C}(\mathcal{Z})$ that
\[
G_{n}(f)+G_{n}(g)=2G_{n}\left(\frac{f+g}{2}\right)+2G_{n}\left(\frac{f-g}{2}\right),
\]
which implies
\[
\frac{G_{n}(f)+G_{n}(g)}{2}-G_{n}\left(\frac{f+g}{2}\right)=G_{n}\left(\frac{f-g}{2}\right).
\]
Since $G_{n}\geq0$, we obtain that $G_{n}$ is convex.
\end{proof}
\begin{rem}
The key to the proof of positivity is the equivalence of positivity and
convexity expressed by the parallelogram identity, which suggests that
$G_{n}(g)$ behaves like a norm in a Hilbert space.
\end{rem}

Using Lemma \ref{lem:PositivityGn} and Proposition
\ref{prop:ConnectionGknToGk+1n}, 
we immediately obtain positivity for $G_{n}^{k}(g)$ and $\Gamma_{n}(g)$.

\begin{thm}
\label{thm:PositivityGknandGamman}Let a sequence of Markov operators
${\bf M}_{n}$ be given such that all ${\bf M}_{n}$ commute with the
semigroup ${\bf T}(t)$. Then, for all $n\in\N$, $k\in\left\{ 0,\dots,n\right\} $
we have that \\$G_{n}^{k}(\oM_{n-k},\oM_{n-k-1},\dots,\oM_{1},g)\geq0$. In particular, 
for all $n\in\N$ we have $
\Gamma_{n}(g)\geq0$.
\end{thm}

\begin{proof}
By Proposition \ref{prop:ConnectionGknToGk+1n}, we have 
\[
\lim_{t\to0}\ \frac{1}{t}\ G_{n}^{k}({\bf T}(t),\oM_{n-k-1},\dots,\oM_{1},f,g)=G_{n}^{k+1}(\oM_{n-k-1},\dots,\oM_{1},f,g)\,.
\]
Since $G_{n}^{0}=G_{n}$ is positive on the diagonal $f=g$ by Lemma
\ref{lem:PositivityGn}, we also obtain that $G_{n}^{k}(g)\geq0$
iteratively for all $k\in\left\{ 1,\dots,n\right\}$. In particular,
we have $G_{n}^{n}(g)=\Gamma_{n}(g)\geq0$.
\end{proof}

\subsection{Associated energies}\label{subsec:Energies}

By Theorem \ref{thm:PositivityGknandGamman} we know that $\Gamma_n(g)$ is positive for all $g\in\mathcal \cD(\oA^n)$. Hence, we may define the associated energies. For $\mu\in\rP(\rZ)$ being the stationary probability measure of $\oA$, i.e. $\oAs \mu=0$, we inductively define the bilinear forms
\begin{align*}
\E_{0}(f,g) & =\langle f\cdot g,\mu\rangle-\langle f,\mu\rangle\langle g,\mu\rangle\\
\E_{n+1}(f,g) & =-\E_{n}(\oA f,g)-\E_{n}(f,\oA g)\,,
\end{align*}
whenever the right-hand side is well-defined. In particular, we have
\begin{align*}
\E_{1}(f,g) & =-\langle\oA f\cdot g,\mu\rangle - \langle \oA g\cdot f,\mu\rangle\, ,
\end{align*}
which is usually called the \textit{Dirichlet form} associated with the Markov generator $\oA$. Moreover, we introduce the notation $\E_n(g):=\E_n(g,g)$.
\begin{lem}
\label{lem:RelationGammanAndEnergies}For all $n\geq 1$ we have that
\[
\E_{n}(f,g)=\langle\Gamma_{n}(f,g),\mu\rangle.
\]
In particular, we have that $\E_n(g)\geq 0$ for all $n\in\N$.
\end{lem}

\begin{proof}
For $n=1$, we have
\begin{align*}
\E_{1}(f,g) & =-\E_{0}(\oA f,g)-\E_{0}(f,\oA g)=\\
 & =-\langle\oA f\cdot g,\mu\rangle+\langle\oA f,\mu\rangle\langle g,\mu\rangle-\langle f\cdot\oA g,\mu\rangle+\langle f,\mu\rangle\langle\oA g,\mu\rangle\\
 & =-\langle\oA f\cdot g+f\cdot\oA g,\mu\rangle=\langle\Gamma_{1}(f,g),\mu\rangle,
\end{align*}
where we have used that $\oAs \mu=0$. This proves the claim for
$n=1$.

Assuming that the claim holds for $n\geq1$ and using the recursion formula, we obtain for $n+1$ that
\begin{align*}
\E_{n+1}(f,g) & =-\E_{n}(\oA f,g)-\E_{n}(f,\oA g)=\\
 & =-\langle\Gamma_{n}(\oA f,g),\mu\rangle-\langle\Gamma_{n}(\oA f,g),\mu\rangle=\\
 & =\langle\oA\Gamma_{n}(f,g)-\Gamma_{n}(\oA f,g)-\langle\Gamma_{n}(\oA f,g),\mu\rangle=\langle\Gamma_{n+1}(f,g),\mu\rangle,
\end{align*}
where we again have used $\oAs \mu=0$. This proves the desired
relation.

Since we have $\E_{n}(g)=\langle\Gamma_{n}(g),\mu\rangle$, positivity follows for $n\geq 1$ by Theorem \ref{thm:PositivityGknandGamman} by using that $\mu \geq 0$. Moreover, $\E_0(g)\geq0$ by Jensen's inequality. 
\end{proof}
One can easily derive an explicit formula for $\E_{n}$.

\begin{prop}
\label{prop:PropertiesEnergies}We have the following explicit form for $n\geq 1$
\begin{equation}
\E_{n}(f,g)=(-1)^{n}\sum_{j=0}^{n}\binom{n}{j}\langle\oA^{n-j}f\cdot\oA^{j}g,\mu\rangle\,.\label{eq:ExplicitFormulae_n}
\end{equation}
\end{prop}

\begin{proof}
We observe that for $n=1$, the right-hand side
is given by \\$-\left(\langle\oA f\cdot g,\mu\rangle + \langle f\cdot\oA
  g,\mu\rangle\right)$, 
which is $\E_{1}$. Assuming that the formula holds true for $n\in\N$,
we obtain, by using $\binom{n+1}{j}=\binom{n}{j-1}+\binom{n}{j}$
that
\begin{align*}
 & (-1)^{n+1}\sum_{j=0}^{n+1}\binom{n+1}{j}\langle\oA^{n+1-j}f\cdot\oA^{j}g,\mu\rangle=\\
 & =(-1)^{n+1}\left\{ \sum_{j=1}^{n+1}\binom{n+1}{j}\langle\oA^{n+1-j}f\cdot\oA^{j}g,\mu\rangle+\langle\oA^{n+1}f\cdot g,\mu\rangle\right\} \\
 & =(-1)^{n+1}\left\{ \sum_{j=1}^{n+1}\binom{n}{j-1}\langle\oA^{n+1-j}f\cdot\oA^{j}g,\mu\rangle+\binom{n}{j}\langle\oA^{n+1-j}f\cdot\oA^{j}g,\mu\rangle+\langle\oA^{n+1}f\cdot g,\mu\rangle\right\} \\
 & =(-1)^{n+1}\left\{ \sum_{k=0}^{n}\binom{n}{k}\langle\oA^{n-k}f\cdot\oA^{k+1}g,\mu\rangle+\sum_{j=0}^{n}\binom{n}{j}\langle\oA^{n+1-j}f\cdot\oA^{j}g,\mu\rangle\right\} \\
 & =-\left\{ \E_{n}(f,\oA g)+\E_{n}(\oA f,g)\right\} =\E_{n+1}(f,g),
\end{align*}
which proves the claimed formula.
\end{proof}
\begin{rem}
Similar explicit expression like
\eqref{eq:ExplicitFormulae_n} 
can also be derived for $\Gamma_{n}$ and $G_{n}^{k}$.
\end{rem}

Considering the solution $g$ of $\dot g=\oA g$, we may define the energy along solutions given by
\[
\e_n(t):=\E_n(g(t),g(t))\,.
\]
The next theorem shows that all energies $\e_n$ decay along solutions and are convex in $t>0$.

\begin{thm}
\label{thm:PropertiesEnergies}
Let $n\in\N$. We have the following properties for
$\e_{n}$:
\begin{enumerate}
\item Along solutions $g=g(t)$ of $\dot{g}=\oA g$ with $g(0)=g_0 \in \rD(\oA^\infty)$ the trajectory  
$t\mapsto \e_{n}(t)$ is differentiable and we have  
\begin{equation}\label{e101}
\frac{\d}{\d t}\e_{n}(t)=-\e_{n+1}(t)\,.
\end{equation}
In particular, $t\mapsto \e_{n}(t)$ is monotonically decreasing and convex.
\item For $n\geq 1$ we have $\lim_{t\to\infty} \e_{n}(t) = 0$ and $\int_0^\infty \e_{n+1}(t)~\d t = \e_{n}(0)$.
\end{enumerate}
\end{thm}
\begin{proof}
We prove both claims separately. For the first claim, we exploit the explicit expression of $\E_{n}$.
Then, we have for all $n\in\N$
\begin{align*}
\frac{\d}{\d t}\e_{n}(t) & =(-1)^{n}\sum_{j=0}^{n}\binom{n}{j}\langle\oA^{n-j}\dot{g}\cdot\oA^{j}g,\mu\rangle+(-1)^{n}\sum_{j=0}^{n}\binom{n}{j}\langle\oA^{n-j}g\cdot\oA^{j}\dot{g},\mu\rangle\\
 & =(-1)^{n}\sum_{j=0}^{n}\binom{n}{j}\langle\oA^{n-j}\oA g\cdot\oA^{j}g,\mu\rangle+(-1)^{n}\sum_{j=0}^{n}\binom{n}{j}\langle\oA^{n-j}g\cdot\oA^{j}\oA g,\mu\rangle=\\
 & =\E_{n}(\oA g,g)+\E_{n}(g,\oA g)=-\E_{n+1}(g(t)) = -\e_{n+1}(t).
\end{align*}
Since $\E_{n+1}$ is positive $\e_n$ decays.  Moreover, the above formula provides that $\frac{\d^2}{\d t^2}\e_{n}(t)=\e_{n+2}(t)\geq 0$, which shows that $\e_n$ is convex. This proves the first claim.

From (\ref{e101}) we have for all $n\in\N$ that
\begin{align*}
\e_{n}(0) - \e_{n}(t) = \int_0^t \e_{n+1}(t') ~\d t'
\end{align*}
Since the left-hand side is monotone and bounded, the limit 
 $\lim_{t\to\infty}\e_{n}(t) =: \e_{n}(\infty)$ exists and we have
\begin{equation}\label{e102}
\e_{n}(0) - \e_{n}(\infty) = \int_0^\infty \e_{n+1}(t') ~\d t'
\end{equation}

This shows the integrability of $t\mapsto \e_{n+1}(t)$ on $[0,\infty)$.  
Since $\e_{n+1}(t)$ is monotone and positive it has to tend to 0. This proves
the second part. 
\end{proof}

In the next section, we show that an explicit quantitative convergence rate can be derived for $\oA$ being a normal operator. For completeness, we recall that $\e_k$ decays exponentially for $k\in\{1,\dots,n\}$ under the assumption of a Poincare-type inequality
$\E_{n+1}(g) \geq c~ \E_{n}(g) $, for fixed $n \geq 1$ with some constant $c>0$.
\begin{cor}\label{cor:ExponentialDecay}
Assume, we have for some fixed $n\in\N$ a Poincare-type inequality
$\E_{n+1}(g) \geq c \,\E_{n}(g) $ with some constant $c>0$. Then, 
$t\mapsto \e_{k}(t)$ decays exponentially in time for all $k \in\{1,\dots,
n\}$ with exponential rate $\exp(-c t)$.
\end{cor}
\begin{proof}
By assumption, we get that $\dot{\e}_{n}(t)\leq-c~ \e_{n}(t)$, which
implies $\e_{n}(t)\leq\e_{n}(0)\mathrm{e}^{-ct}$, the exponential
decay of $\e_{n}$. Integrating $\dot{\e}_{n-1}(t)=-\e_{n}(t)$ in time, we conclude
\[
\e_{n-1}(T)-\e_{n-1}(t)=\int_{t}^{T}-\e_{n}(s)\d
s\geq-\int_{t}^{T}\e_{n}(0)\mathrm{e}^{-c s}\d
s=\frac{\e_{n}(0)}{c}\left(\mathrm{e}^{-c T}-\mathrm{e}^{-c t}\right)\,.
\]
Taking the limit $T\to\infty$ and using that $\e_{n-1}(T=\infty)=0$
by Theorem \ref{thm:PolynomialDecayEnergy}, we conclude that
\[
\e_{n-1}(t)\leq\frac{\e_{n}(0)}{c}\mathrm{e}^{-c t}\,.
\]
Clearly this can be iterated to show that $t\mapsto\e_{k}(t)$ decays
exponentially for all $k\in\left\{ 1,\dots,n\right\} $.
\end{proof}


\begin{rem}
We remark that the presented theory is valid for arbitrary Markov
generators. However, there are cases for which the discussed inequalities are
trivially satisfied but do not contain any
information.   

The first obvious case is when parts of the spectrum lie on the imaginary axis
(without considering 0) which is true for derivations. 
A Markov generator $\oA$ is called {\it derivation}, if $\rD(\oA)$ is a
subalgebra of $\rC$, $\eins \in \rD(\oA)$, $\oA \eins = 0$ and
$\oA(f \cdot g) = f \cdot \oA g + g \cdot \oA f,~f,g \in \rD(\oA)$
holds. Then we have $\Gamma_1(f,g)
\equiv 0$ what implies $\Gamma_n(f,g)\equiv 0$ for all $n \geq 1$.

A second case is a degenerated generator $\oA$.  A Markov
generator $\oA$ is called {\it degenerate} if there exists a closed set $B
\subset 
\rZ$ and $(\oA g)(z)=0$ holds for all $z \in B$ and $g \in \rD(\oA)$. 
Obviously, any measure $\mu \in \rP$ with $\mu(B)=1$ is a stationary measure
for $\oA^*$. 
Although in this case $\Gamma_n \not\equiv 0$ holds, we obtain 
$\E_n(f,g)=\la \Gamma_n(f,g) , \mu \ra =0$ for all $n \geq 1$.
\end{rem}

\section{Decay rate for energies in Hilbert space}

So far the decay of the energies $t\mapsto\e_{n}(t)=\E_{n}(g(t))$ for a general Markov generator $\oA$
along solution is not quantified. It is well-known that whenever the
generator $\oA$ has a spectral gap, i.e. the largest non-trivial
eigenvalue has strictly negative real part, then exponential decay
of the solution and, hence, also for energies should hold in theory. For a general
Markov generator $\oA$ the real-parts of its eigenvalues may accumulate
in zero, i.e. there is no spectral gap and no exponential decay. Moreover,
in practice it is rather difficult to compute the spectral gap for
a given Markov generator.

As we will see it is possible to derive polynomial ($n$-dependent)
decay for all energies $t\mapsto\e_{n}(t)$ for general $\oA$, which satisfy
an operator-theoretic normality condition. Because \textit{normality}
(like \textit{self-adjointness}) are terms for operators in Hilbert spaces,
we first lift $\oA$ to the natural Hilbert space $\L^{2}(\mu)$.
Then in Section \ref{subsec:Log-convexity-in-n},
we derive log-convexity for the sequence $(\E_{n}(g))_{n\in\N}$ which enables to
derive conclusion for $t\mapsto\e_{n}(t)$.

\subsection{Hilbert space embedding of $\oA$}

The generator $(\oA,\mathcal{D}(\oA))$ with semigroup ${\bf T}(t)$
with stationary measure $\mu\in\mathcal{P}(\mathcal{Z})$ is
defined on $\mathcal{C}(\mathcal{Z})$. 
Certain important properties, however, are naturally to be studied in a
Hilbert space.  

As usual we define the real
separable Hilbert space $\L^{2}(\mu)$ as the completion
of $\mathcal{C}(\mathcal{Z})$ 
with the norm (written $\|g\|_\mu$)
induced by the scalar product
\[
(f,g)_{\mu}:=\langle f\cdot g,\mu\rangle = \int_\rZ f(z)g(z)p(\d z) \,.
\]
By definition $\mathcal{C}(\mathcal{Z})\subset\L^{2}(\mu)$ is
dense\footnote{The choice of $\L^{2}(\mu)$ instead of any $\L^{2}(p)$ with an
  arbitrary positive 
measure $p$, has good reasons. In order to transfer
the theory in $\rC$ to $\L^2$, certain properties should be preserved as  boundedness and strong-continuity of
$\oT(t)$. We note that in general a bounded operator $\oM$ on $\mathcal{C}(\mathcal{Z})$
is not bounded on $\L^{2}(p)$. Take for example $\mathcal{Z}=[0,1]$,
$\oM g(z):=g(z_{0})$ and $p$ the Lebesgue measure on $[0,1]$. The distinguished role of the stationary measure $\mu$ becomes clear when considering, for example, symmetry of $\oA$, i.e.  $(\oA f,g)_{p} = (f,\oA g)_{p}$ or, equivalently,
$\la g \cdot \oA f,p \ra = \la f \cdot \oA g,p \ra $. Setting $f=\eins$ we conclude $\oA^* p=0$, so $p$ must then be a stationary measure of $\oA$.\label{foot1}}.
We implicitly use this fact in the following, writing for example
$(f,g)_{\mu}=\langle f\cdot g,\mu\rangle$, 
although this identity is valid only for $f,g \in \rC(\rZ) \subset \L^2(\mu)$.

Recall, that if $\mu$ is the invariant measure of the Markov operator $\oM$,
i.e. $\oM^{*}\mu=\mu$, then $\oM$ can be boundedly extended to $\L^{2}(\mu)$.
To see this, we observe for $g\in\L^{2}(\mu)$ that
\[
  \|\oM g\|_\mu^2 =
(\oM g,\oM g)_{\mu}=\langle\left(\oM g\right)^{2},\mu\rangle\leq\langle\oM g^{2},\mu\rangle=\langle g^{2},\oM^{*}\mu\rangle=\langle g^{2},\mu\rangle=(g,g)_{\mu},
\]
where we have used Jensen's inequality and that $\oM^{*}\mu=\mu$.
Hence, $\oM$ is bounded on $\L^{2}(\mu)$ with constant 1. In the
following we also denote the operator on the larger space $\L^{2}(\mu)$
by the same symbol.

Similarly, a strongly-continuous semigroup of operators $({\bf T}(t))_{t\geq 0}$
on $\mathcal{C}(\mathcal{Z})$ with invariant measure
$\mu\in\mathcal{P}(\mathcal{Z})$
can be extended to the space $\L^{2}(\mu)$. Clearly, then the family
$({\bf T}(t))_{t\geq 0}$ is a semigroup of bounded operators on $\L^{2}(\mu)$. Moreover, $({\bf T}(t))_{t\geq 0}$ is also strongly-continuous
in $\L^2(\mu)$, because we have
  \begin{align*}
  \|\oT(t) &g - g\|_\mu^2 =
  \la (\oT(t) g - g)\cdot (\oT(t) g - g),\mu \ra =
  \la (\oT(t) g)^2,\mu \ra - \la  g^2,\mu \ra 
  - 2\la g \cdot (\oT(t) g - g),\mu \ra \\&\leq
  \la \oT(t) g^2,\mu \ra - \la  g^2,\mu \ra 
  - 2\la g \cdot (\oT(t) g - g),\mu \ra =
    \la \oT(t) g^2 - g^2,\mu \ra - 2\la g \cdot (\oT(t) g - g),\mu \ra.
    \end{align*}
The right hand side tends to 0 for $t \to 0$ due to the 
strong continuity of $\oT(t)$ in $\rC$. The generator of $({\bf T}(t))_{t\geq 0}$ on $\L^{2}(\mu)$ is
(with a slight abuse of notation) also denoted by
$\left(\oA,\mathcal{D}(\oA)\right)$, $\mathcal{D}(\oA)\subset\L^{2}(\mu)$.
It is closed and densely defined on $\L^{2}(\mu)$. There will be
no confusion with the notation because in this section we are only
interested in the Hilbert space $\L^{2}(\mu)$-version. The generator
$\oA$ coincides with the original generator on $\mathcal{C}(\mathcal{Z})$
because the semigroups coincide there. Since ${\bf T}(t)$ is a contraction,
the spectrum of $\oA$ is located on the left-hand side of the complex
plane.

\subsection{Normality of ${\bf A}$ and the connection to $\E_{n}$}

To quantify the decay rate of the energies $\e_{n}$, we assume that
${\bf A}$ is a normal operator. For this, we first define its
$\L^{2}(\mu)$-adjoint
$\left(\oAsh ,\mathcal{D}(\oAsh )\right)$ as usual\footnote{Note that we use a star ($\star$) to distinguish the Hilbert-space
adjoint $\oAsh $ with the Banach space dual $\oAs $,
where the latter is defined on $\mathcal{C}^{*}(\mathcal{Z})$. } (we refer e.g. to \cite{Schm12USOHS} for unbounded operators on Hilbert
spaces). The
operator $\oAsh $ is well defined since ${\bf A}$ is densely
defined, and moreover, it is closed. In the following we make the
following assumption.
\begin{assumption*}
The operator $({\bf A},\mathcal{D}({\bf A}))$ is a normal operator on $\L^{2}(\mu)$,
meaning that $\mathcal{D}({\bf A})=\mathcal{D}(\oAsh )$
and $\|{\bf A}g\|_\mu=\|\oAsh g\|_\mu$ for all $g\in\mathcal{D}({\bf A})=\mathcal{D}(\oAsh )$.
\end{assumption*}
Since $\oA$ is closed, we have $\oAsh {\bf A}={\bf A}\oAsh $.
We note that normality is a substantially more general concept than
self-adjointness,  
which would mean that $\mathcal{D}({\bf A})=\mathcal{D}(\oAsh )$
and ${\bf A}=\oAsh $. In Section \ref{sec:Discussion-of-normality},
we discuss what normality of ${\bf A}$ means for the original operator
on $\mathcal{C}(\mathcal{Z})$. In particular we provide an operator-theoretic
version for normality with respect to $\mathcal{C}(\mathcal{Z})$.
Moreover, we show with an example that without the assumption log-convexity may fail, which is our main ingredient to derive polynomial decay of the energy.

For the normal operator ${\bf A}$, we define the operator
\begin{align}
\label{eq:C}
{\bf C}:=-\left({\bf A}+\oAsh \right),\qquad\mathcal{D}({\bf C})=\mathcal{D}({\bf A})=\mathcal{D}(\oAsh )\,.
\end{align}
In particular, normality of ${\bf A}$ implies that $\mathcal{D}({\bf C})=\mathcal{D}({\bf A})$
and that ${\bf C}$ is densely defined. Moreover, we see that for
all $f,g\in\mathcal{D}({\bf C})=\mathcal{D}({\bf A})$, we have 
\[
(f,{\bf C}g)_{\mu}=-(f,\left({\bf A}+\oAsh \right)g)_{\mu}=-(f,{\bf A}g)_{\mu}-({\bf A}f,g)_{\mu},
\]
which shows that ${\bf C}$ is symmetric and, hence, also closable.
In general ${\bf C}$ is, as a sum of a normal operator with its adjoint,
not necessarily closed\footnote{As an example consider ${\bf A}=\frac{\d}{\d x}$ on $\L^{2}(\R)$.
Then ${\bf A}$ is normal on $\mathcal{D}({\bf A})=\H^{1}(\R)$ because
$\oAsh =-\frac{\d}{\d x}=-{\bf A}$. But ${\bf A}+\oAsh ={\bf 0}$,
which is not closed as defined on $\H^{1}(\R)$. See e.g. \cite{ArTr20EPDI} for recent results on the variety of domain intersection for an operator with its adjoint}. However classical spectral theory for normal operators provides that the closure
$\overline{{\bf C}}=-\overline{\left({\bf A}+\oAsh \right)}$
is self-adjoint, see e.g. \cite{Schm12USOHS}. In fact the
complex spectral family for a (in general unbounded) normal operator
can be split into two real spectral families, which provides a decomposition
into a real and imaginary part of the normal operator that are essentially
self-adjoint. Here, we are not interested in spectral properties of
${\bf A}$, and, in particular, we do not make any assumptions on
the spectrum of ${\bf A}$, we only use the fact that ${\bf C}$ is
essentially self-adjoint.

Moreover, ${\bf C}$ as well as its closure $\overline{{\bf C}}$
is positive in the form sense, because
\[
(g,{\bf C}g)_{\mu}=-(g,\oA g)_{\mu}-(g,\oA g)_{\mu}=-2\langle g\cdot\oA g,\mu\rangle=\E_{1}(g,g)\geq0\,,
\]
holds for all $g\in\mathcal{D}({\bf C})$. Summarizing, ${\bf C}$ is a
positive essentially self-adjoint operator. In particular, there is
a unique positive self-adjoint square-root ${\bf B}$ of $\overline{{\bf C}}$ (see e.g. \cite{Kato95PTLO}),
which will be used later.

As it turns out powers of the operator ${\bf C}$ are directly related
to the energies $\E_{n}(f,g)$. Since ${\bf A}$ is normal, we immediately
obtain that the operators ${\bf C}$ and ${\bf A}$ commute on $\mathcal{D}({\bf A}^{2})$.
Hence, also higher powers of ${\bf C}$ commute with ${\bf A}$. This
is used to express $\E_{n}$ in terms of powers of ${\bf C}$. We remark that an analogous statement is well known and commonly used for symmetric operators (see, e.g., \eqref{e002} in the Introduction)
\begin{prop}
Let $({\bf A},\rD(\oA))$ be normal and $\oC$ be defined by \eqref{eq:C}. Then for $n\geq1$ we have
\[
\E_{n}(f,g)=\left(f,\oC^{n}g\right)_{\mu},
\]
whenever the right-hand side is bounded (i.e. for all $f,g\in\mathcal{D}({\bf A}^{n})$).
\end{prop}

\begin{proof}
We prove that the sequence, $\left\{ \left(f,\oC^{n}g\right)_{\mu}\right\} _{n\in\mathbb{N}}$
satisfies the same recursion formula as $\E_{n}(f,g)$. For $n=1$,
we have already seen that $\left(f,\oC g\right)_{\mu}=-\left(f,\oA g\right)_{\mu}-\left(\oA f,g\right)_{\mu}=\E_{1}(f,g)$.

For $n\geq2$, we observe that for $f,g\in\mathcal{D}({\bf A}^{n+1})$:
\begin{align*}
\left(f,\oC^{n+1}g\right)_{\mu} & =\left(f,\oC\oC^{n}g\right)_{\mu}=-\left(f,\oA\oC^{n}g\right)_{\mu}-\left(\oA f,\oC^{n}g\right)_{\mu}=-\left(f,\oC^{n}\oA g\right)_{\mu}-\left(\oA f,\oC^{n}g\right)_{\mu}\\
 & =-\E_{n}(f,\oA g)-\E_{n}(\oA f,g)=\E_{n+1}(f,g)\,,
\end{align*}
where we have used that $\oA$ and $\oC^n$ commute on $\rD(\oA^{n+1})$.
\end{proof}
\begin{rem}
We remark that the above formula makes also sense to define fractional powers of the energies $\E_\alpha$ via $\overline{\oC}^\alpha$ for $\alpha>0$ by interpolation. This will be investigated in subsequent work.
\end{rem}

\subsection{Log-convexity in
  $n$ and polynomial decay}\label{subsec:Log-convexity-in-n}

With the explicit and closed form of the energies $\E_{n}(f,g)=\left(f,\oC^{n}g\right)_{\mu}$
we are able to show that $t\mapsto\e_{n}(t)$ decays with polynomial
rate. To see this we show that the sequence $(\E_{n}(g))_{n\in\N}$ is
log-convex for any $g\in\rD(\oA^\infty)$.
\begin{prop}
\label{prop:LogConvexityOfSequenceOfEnergies}Let $({\bf A},\mathcal{D}({\bf A}))$
be a normal operator and let $g\in\mathcal{D}({\bf A}^{\infty})$. Then:
\begin{enumerate}
\item The sequence $\left(\E_{n}(g)\right)_{n\in\mathbb{N}}$ is log-convex,
i.e. for all $n\in\mathbb{N}$ we have $\E_{n+1}(g)\E_{n-1}(g)\geq\E_{n}(g)^{2}$.
\item If $\E_{0}(g)>0$, then we have $\left(\frac{\E_{n+1}(g)}{\E_{0}(g)}\right)^{\frac{1}{n+1}}\geq\left(\frac{\E_{n}(g)}{\E_{0}(g)}\right)^{\frac{1}{n}}$.
\end{enumerate}
\end{prop}

\begin{proof}
Since ${\bf A}$ is normal, the operator $\left(\oC,\mathcal{D}(\oC)\right)$ defined by \eqref{eq:C}
is positive and essentially self-adjoint. Let ${\bf B}$ be the positive
and self-adjoint square-root of $\overline{{\bf C}}$ on $\L^{2}(\mu)$,
i.e. we have $\oB^{2}g=\overline{\oC}g$ for all $g\in\mathcal{D}(\overline{{\bf C}})$.
Then, for $\alpha,\beta\in\N$
and $g\in\mathcal{D}({\bf B}^{2\max\{\alpha,\beta\}})$ we have that 
\begin{align*}
0\leq & \int_{\mathcal{Z}}\int_{\mathcal{Z}}\Big[\big(\oB^{\alpha}g\big)(z)\big(\oB^{\beta}g\big)(z')-\big(\oB^{\alpha}g\big)(z')\big(\oB^{\beta}g\big)(z)\Big]^{2}\mu(\d z)\mu(\d z')\\
 & =\left[\int\big(\oB^{\alpha}g\big)^{2}(z)\mu(\d z)\cdot\int\big(\oB^{\beta}g\big)^{2}(z')\mu(\d z')+\int\big(\oB^{\alpha}g\big)^{2}(z')\mu(\d z')\cdot\int\big(\oB^{\beta}g\big)^{2}(z)\mu(\d z)\right.\\
 & \qquad -\left.2\int\big(\oB^{\alpha}g\big)(z)\big(\oB^{\beta}g\big)(z)\mu(\d z)\cdot\int\big(\oB^{\alpha}g\big)(z')\big(\oB^{\beta}g\big)(z')\mu(\d z')\right]\\
 & =(\oB^{\alpha}g,\oB^{\alpha}g)_{\mu}(\oB^{\beta}g,\oB^{\beta}g)_{\mu}-(\oB^{\alpha}g,\oB^{\beta}g)_{\mu}^{2} = (g,\oB^{2\alpha}g)_{\mu}(g,\oB^{2\beta}g)_{\mu}-(g,\oB^{\alpha+\beta}g)_{\mu}^{2}\,.
\end{align*}
Setting, $\alpha=n+1$ and $\beta=n-1$ and using $\oB^{2}g=\overline{\oC}g={\bf C}g$
for $g\in\mathcal{D}({\bf C}^{n+1})$, we obtain
\begin{align*}
0 & \leq(g,\oB^{2(n+1)}g)_{\mu}(g,\oB^{2(n-1)}g)_{\mu}-(g,\oB^{2n}g)_{\mu}^{2}\\
 & =(g,\oC^{n+1}g)_{\mu}(g,\oC^{n-1}g)_{\mu}-(g,\oC^{n}g)_{\mu}^{2}=\E_{n+1}(g)\E_{n-1}(g)-\E_{n}(g)^{2}\,,
\end{align*}
which proves the first claim.

For the second claim, we use the fact that for a given non-negative log-convex
sequence $\left(a_{n}\right)_{n\in\mathbb{N}}$ 
with $a_{0}>0$, we have $\left(\frac{a_{n+1}}{a_{0}}\right)^{\frac{1}{n+1}}\geq\left(\frac{a_{n}}{a_{0}}\right)^{\frac{1}{n}}$
for all $n\geq1$. A proof\footnote{The historic proof is from 1729 and goes back to C. Maclaurin.} for this can be found for example in \cite{BecBel61I}, which is given here for completeness.

Assuming that $\left(a_{n}\right)_{n\in\mathbb{N}}$ is a non-negative log-convex
sequence with $a_{0}>0$, we have 
\begin{align*}
a_1^2  \leq  a_0 a_2,\quad  a_2^2  \leq  a_1 a_3, \quad  a_3^2  \leq  a_2 a_4, \quad 
...,\quad  a_n^2  \leq  a_{n-1} a_{n+1} \,.
\end{align*}
If we take the $j$-th inequalities to the $j$-th power
and multiply all inequalities, we get
\[
a_1^2a_2^4a_3^6 \cdots a_{n-1}^{2n-2} a_n^{2n} \leq 
(a_0 a_2)^1  (a_1 a_3)^2  (a_2 a_4)^3 \cdots (a_{n-2} a_{n})^{n-1}  (a_{n-1}
a_{n+1})^n .
\]
This can be simplified to $
a_n^{2n} \leq a_0 a_n^{n-1} a_{n+1}^n$, which implies $
(a_n/a_0)^{1 \over n} \leq (a_{n+1}/a_0)^{1 \over n+1}$.

\end{proof}

With the help of the log-convexity of the sequence $\left(\E_{n}(g)\right)_{n\in\mathbb{N}}$,
we can derive asymptotic polynomial decay for $t\mapsto\e_{n}(t)$ as an upper bound.
\begin{thm}
\label{thm:PolynomialDecayEnergy}Let $g$ be a solution of $\dot{g}=\oA g$ with $g(0)=g_0\in\rD(\oA^\infty)$
such that $\e_{0}(0)>0$. Moreover, let the generator $\oA$ be normal.
If $\e_{n}(0)>0,$ then we have $$\e_{n}(t)\leq\left(\e_{n}(0)^{-1/n}+\frac{t}{n}\e_{0}(0)^{-1/n}\right)^{-n}.$$
In particular, $t\mapsto\e_{n}(t)$ decays to zero with polynomial
rate $O(t^{-n})$.
\end{thm}

\begin{proof}
Using that $g=g(t)$ is solution, we have that
\[
\frac{\d}{\d t}\e_{n}(g)=\dot{\e}_{n}(t)=-\e_{n+1}(t)\,.
\]
By Proposition \ref{prop:LogConvexityOfSequenceOfEnergies} we have $\E_{n+1}(g)\geq\E_{0}(g)^{-1/n}\E_{n}(g)^\frac{n+1}{n}$
for all $n\geq1$. Using
$\e_{0}(0)\geq\e_{0}(t)$, we obtain that
\begin{equation}
\dot{\e}_{n}(t)\leq-\e_{0}(t)^{-1/n}\e_{n}(t)^{\frac{n+1}{n}}\leq-\e_{0}(0)^{-1/n}\e_{n}(t)^{\frac{n+1}{n}}\,.\label{eq:DifferentialInequalityEn}
\end{equation}
From this differential inequality we conclude that $\e_{n}(t)$ satisfies
the desired estimate. Indeed, we may assume that $\e_{n}(t)>0$ for
all $t>0$, otherwise the claim is trivial. Introducing 
\[
\alpha:=\e_{0}(0)^{-1/n},\quad X(\e):=\frac{n}{\alpha}\e^{-1/n},\quad x(t):=X(\e_{n}(t))-X(\e_{n}(0))-t,
\]
we get $x(0)=0$ and 
\[
\dot{x}(t)=X'(\e_{n}(t))\dot{\e}_{n}(t)-1=-\frac{1}{\alpha}\e_{n}(t)^{-1-1/n}\dot{\e_{n}}(t)-1\geq0,
\]
by \eqref{eq:DifferentialInequalityEn}. Hence, we have $x(t)\geq0$ for $t\geq0$
which means $X(\e_{n}(t))\geq X(\e_{n}(0))+t$. Inserting the
definition of $X$, we obtain
\[
\frac{n}{\alpha}\e_{n}(t)^{-1/n}\geq\frac{n}{\alpha}\e_{n}(0)^{-1/n}+t\quad\Rightarrow\quad\e_{n}(t)\leq\left(\e_{n}(0)^{-1/n}+\e_{0}(0)^{-1/n}\frac{t}{n}\right)^{-n},
\]
which we wanted to show.
\end{proof}

\begin{rem}
Crucial for Theorem \ref{thm:PolynomialDecayEnergy} was the log-convexity
from Proposition \ref{prop:LogConvexityOfSequenceOfEnergies} of the
form $\frac{\E_{n+1}(g)}{\E_{0}(g)}\geq\left(\frac{\E_{n}(g)}{\E_{0}(g)}\right)^{\frac{n+1}{n}}$.
Heuristically, the exponent on the right-hand side converges to 1
as $n\to\infty$, meaning that the above estimate becomes more and
more a linear inequality with increasing $n\in\N$. This means that
at least in theory it becomes easier to prove a linear inequality
for larger $n\in\N$, and to proceed as in Corollary \ref{cor:ExponentialDecay} to obtain exponential decay.
\end{rem}
\begin{rem}
We remark, that for a general Markov generator $\oA$ no exponential decay is to be expected, because, in principle, the real parts of the spectrum of $\oA$ may have a clustering point in 0, or, in other words, no spectral gap is present. 

Although it is not easy to find examples for this and to make the above polynomial estimates explicit, we provide a well-known example, namely diffusion on the real line. However, we note that the example does not fit exactly into the presented theory since it considers a non-compact domain. Considering the diffusion equation $u_t = {D \over 2} u_{xx}$
on $\R$, the fundamental solution is given by the Gaussian
kernel $u={1 \over \sqrt{2 \pi D t}} \mathrm{e}^{-{x^2 \over 2 D t}}$, where the stationary measure is the Lebesgue measure $\mu=\d x$. It is well-known that whole negative real axis belongs to the continuous spectrum.   It is easy to calculate that
\begin{align*}
\e_0(t) &= 1, \quad &\e_1(t) &= -2 ( u,\oA u)_\mu = {1 \over 4 \sqrt{\pi D t^3}},\\
\e_2(t) &= 4 ( u,\oA^2 u))_\mu = {3 \over 8 \sqrt{\pi D t^5}}, \quad 
&\e_3(t) &= -8 ( u,\oA^3 u))_\mu = {3\cdot 5 \over 16 \sqrt{\pi D t^7}}, \\
\e_4(t) &= 16 ( u,\oA^4 u))_\mu =
            {3\cdot 5\cdot 7 \over 32 \sqrt{\pi D t^9}} ,&\e_5(t) &= -32 ( u,\oA^5 u))_\mu =
            {3\cdot 5\cdot 7 \cdot 9\over 64 \sqrt{\pi D t^{11}}} \,.
\end{align*}
\end{rem}

\subsection{\label{subsec:Log-convexity-in-time}Log-convexity  in time}

The log-convexity of the sequence $(\E_n(g))_n\in\N$ provides also log-convexity of the trajectory of energies $t\mapsto\e_n(t)$, which is a stronger statement than the convexity of $t\mapsto\e_n(t)$ from Theorem \ref{thm:PropertiesEnergies}. Moreover, log-convexity provides another interesting feature, namely that $t\mapsto\big(\e_n(t)/\e_n(0)\big)^{1/t}$ is increasing in time. This can be understood an a complementary information to the asymptotic decay of $t\mapsto\big(\e_n(t)/\e_n(0)\big)$ as in Theorem \ref{thm:PolynomialDecayEnergy}.
\begin{thm}\label{thm:LogConvexityInTime}
Let $g$ be a solution of $\dot{g}=\oA g$  with $g(0)=g_0\in\rD(\oA^\infty)$
such that for fixed $n\in\N$ we have $\e_n(t)>0$ for all $t>0$. Moreover, let $\oA$ be normal. Then we have the following:
\begin{enumerate}
\item The energy trajectory $t\mapsto\e_{n}(t)$ is log-convex in time, i.e., $\frac{\d^{2}}{\d t^{2}}\log\e_{n}(t)\geq0$.
\item The trajectory $t\mapsto\left(\frac{\e_{n}(t)}{\e_n{0}}\right)^{1/t}$ is increasing in time, and, moreover, it converges.
\end{enumerate}
\end{thm}

\begin{proof}
By log-convexity with respect to $n$ we have
\begin{align*}
\ddot{\e}_n(t) &= - \dot{\e}_{n+1}(t) =  \e_{{n+2}}(t) \geq
\e_{n}^{-1}(t) \e_{{n+1}}^2(t) = \e_{n}^{-1}(t) \dot{\e}_{n}^2(t), 
\end{align*}
where we have used that $\e_n(t)$ is strictly positive for all $t>0$. Hence, we conclude
\begin{align*}
\frac{\d^{2}}{\d t^{2}} \log \e_n(t) =\frac{\d}{\d t} \frac{\dot{\e}_n(t)}{\e_n(t)} =\frac{  \e_{n}(t) \ddot{\e}_n(t) -  \dot{\e}_n^2(t)}{
  \e_{n}^2(t)} \geq 0.
\end{align*}
This proves the first claim.

For the second claim, we define the function $f_n(t)=\frac{\e_n(t)}{\e_n(0)}$. Then, $f_n$ is also log-convex in time, positive, and we have $f_n(0)=1$. Using the following identity
\begin{align*}
\frac{\d}{\d t}\bigg(\frac 1 t \log f_n(t)\bigg) = \frac{1}{t^2}\int_0^t s ~\frac{\d^2}{\d s^2}\bigg(\log f_n(s)\bigg)~\d s
\end{align*}
which can be easily checked to hold for all $t>0$ and any positive $C^1$-function $f_n$ satisfying $f_n(0)=1$, we have that 
\begin{align*}
0\leq \frac{\d}{\d t}\bigg(\frac 1 t \log f_n(t)\bigg) = \frac{\d}{\d t}\bigg( \log f_n(t)^{1/t}\bigg).
\end{align*}
Hence, $t\mapsto \log f_n(t)^{1/t}$ is increasing which implies that also $t\mapsto f_n(t)^{1/t} = \big(\frac{\e_n(t)}{\e_n(0)}\big)^{1/t}$ is increasing.

To see that $f_n$ converges, we use the polynomial bound from Theorem \ref{thm:PolynomialDecayEnergy}. We have
\begin{align*}
f_n(t) &= \frac{\e_n(t)}{\e_n(0)} \leq \left(1+\bigg(\frac{\e_{0}(0)}{\e_{n}(0)}\bigg)^{-1/n}\frac{t}{n}\right)^{-n}\\
\Rightarrow f_n(t)^{1/t} &= \bigg(\frac{\e_n(t)}{\e_n(0)} \bigg)^{1/t} \leq \left(1+\bigg(\frac{\e_{0}(0)}{\e_{n}(0)}\bigg)^{-1/n}\frac{t}{n}\right)^{-n/t}\leq 1.
\end{align*}
In the last estimate we used that $(1+\alpha s)^{-1/s}\leq 1$ for all $s>0$ and $\alpha>0$. This completes the proof of the second claim.
\end{proof}
Since $\e_{n}(0)^{1/t}$ converges to 1, we obviously obtain that $\e_{n}(t)^{1/t}$ converges as $t\to \infty$. We finally remark that the above formulas show clearly the connection of the limit $\lim_{t\to\infty} f_n(t)^{1/t} =: \mathrm{e}^{-\lambda}$, $\lambda\in[0,\infty[$ to the spectral gap of $\oA$.

\subsection{Discussion of normality\label{sec:Discussion-of-normality}}

In this section we discuss the assumption that the $\L^{2}(\mu)$-version
of ${\bf A}$ is normal. In particular, we are interested how normality
transfers to operators on the space $\mathcal{C}(\mathcal{Z})$.
Moreover, we will see that if ${\bf A}$ is not normal in $\L^{2}(\mu)$
then the sequence $(\E_{n}(g))_{n\in\N}$ is not log-convex in general.

From operator-theoretic perspective, normality for an unbounded
operator is
the canonical property of an operator in Hilbert space. Roughly speaking, a
normal operator is just as diagonalizable as a self-adjoint operator, but may
have complex spectrum.  Normal operators behave to self-adjoint ones like
complex numbers to real ones. One can show that for a generic Markov operator there exists a Hilbert space
in which its extension is normal (see \cite{SteSte21SMOGHST} where for general
Markov operators appropriate Hilbert spaces are constructed). The Hilbert space is not necessarily $\L^{2}(\mu)$ suggesting that lack of normality of an operator indicates an incorrectly
chosen Hilbert space (compare also with the footnote on p. \pageref{foot1}).

To understand what $\L^{2}(\mu)$-normality means on $\rC(\rZ)$, we make the technical assumption that all involved operators
are bounded (i.e. neglecting non-trivial domain issues), although much
reasoning generalizes to unbounded operators. We define the multiplication operator ${\bf Q}_{\mu}$
given by
\[
{\bf Q}_{\mu}:\mathcal{C}(\mathcal{Z})\to\mathcal{C}^{*}(\mathcal{Z}),\quad\forall f,g\in\mathcal{C}(\mathcal{Z})\ :\ \langle f,{\bf Q}_{\mu}g\rangle=\langle f\cdot g,\mu\rangle\,.
\]
Then ${\bf Q}_{\mu}$ is symmetric and positive. With this we may
express the $\L^{2}(\mu)$-adjoint $\oAsh $ of ${\bf A}$
in term of the Banach-space dual $\oAs $: The $\L^{2}(\mu)$-adjoint
is given by $(f,\oAsh g)_{\mu}=({\bf A}f,g)_{\mu}$, which
is, by definition, equivalent to 
\begin{align*}
\langle f\cdot\oAsh g,\mu\rangle & =\langle{\bf A}f\cdot g,\mu\rangle\ \Leftrightarrow\ \langle f,{\bf Q}_{\mu}\oAsh g\rangle=\langle{\bf A}f,{\bf Q}_{\mu}g\rangle=\langle f,\oAs {\bf Q}_{\mu}g\rangle.
\end{align*}
Since $f,g$ are arbitrary,
we conclude that the $\L^{2}(\mu)$-adjoint $\oAsh$ has to satisfy $
{\bf Q}_{\mu}\oAsh=\oAs {\bf Q}_{\mu}$.
Assuming that there is a Markov generator $\oX$ solving the operator equation
$\oQ_\mu \oX = \oAs \oQ_\mu$ as an equation for operators $\rC(\rZ)
  \to \rC^*(\rZ)$, we observe that $\oAsh$ is the $\L^2(\mu)$-extension of
$\oX$. If, 
moreover, $\mu$ is positive (i.e. $\mu(U) > 0$ for any open set $U \subset \rZ$), then on the range of
$\oAs \oQ_\mu$, $\oQ_\mu$ is one-to-one\footnote{The function $\oQ_\mu^{-1}p$ can be understood as a Radon-Nikodym derivative of $p\in\rC^*(\rZ)$ with respect to $\mu$.} and we have $\oX = \oQ_\mu^{-1} \oAs
\oQ_\mu$. For details, we refer to \cite{Step22PIMO}.

Regarding symmetry, we observe that ${\bf A}$ is $\L^{2}(\mu)$-self-adjoint
if and only if ${\bf Q}_{\mu}{\bf A}=\oAs {\bf Q}_{\mu}$, which for stochastic
processes is usually called \textit{detailed balance}, or that $\oAs$ generates a reversible process. The normality condition can be expressed by $\oX\oA=\oA\oX$, or equivalently, by ${\bf A}{\bf Q}_{\mu}^{-1}\oAs {\bf Q}_{\mu}={\bf Q}_{\mu}^{-1}\oAs {\bf Q}_{\mu}{\bf A}$.
We discuss these two conditions with a finite dimensional example.

Let us consider on $\mathcal{Z}=\left\{ 1,2,3\right\} $ a Markov
generator given by
\[
\oA=\begin{pmatrix}-a & a & 0\\
0 & -b & b\\
c & 0 & -c
\end{pmatrix},
\]
which describes the exchange of mass along a loop with rates $a,b,c>0$.
The stationary measure is proportional to $\mu=\left(\frac{1}{a},\frac{1}{b},\frac{1}{c}\right)$,
i.e. $\oAs \mu=0$. The multiplication operator ${\bf Q}_{\mu}$
is a diagonal operator given by $\mathrm{diag}(\mu)$. An easy calculation shows that the $\L^{2}(\mu)$-adjoint $\oAsh$
is then given by 
\[
\oAsh =\begin{pmatrix}-a & 0 & a\\
b & -b & 0\\
0 & c & -c
\end{pmatrix},
\]
describing the exchange of mass along the reverse loop. We observe
that ${\bf A}\neq\oAsh $, i.e. ${\bf A}$ never satisfies
detailed balance.

Moreover, we can easily verify that $\oA$ and $\oAsh$
commute (i.e. $\oA$ is normal) if and only if $a=b=c$. In this situation, Theorem \ref{thm:PolynomialDecayEnergy}
is applicable and provides the polynomial decay. Of course, along
solutions $\mathrm{e}^{t{\bf A}}$ the energies $\e_{n}$ will decay also exponentially fast with rate given by the real part of the first non-trivial eigenvalue of ${\bf A}$, which is $-\frac{3}{2}a$.

Finally, we observe that log-convexity for the sequence $(\E_n(g))_{n\in\N}$
does not hold if ${\bf A}$ and $\oAsh $ do not commute.
To see this, we set $a=4$, $b=c=1$, compute $\E_{0},\E_{1},\E_{2}$ and evaluate $\E_{2}\E_{0}-\E_{1}^{2}$
for $g_\alpha=(1,2\alpha,4\alpha)$ which is
\[
\E_{2}(g_\alpha)\E_{0}(g_\alpha)-\E_{1}^{2}(g_\alpha) = \frac 8 3 (1-3\alpha)\alpha.
\]
We see that this is either positive or negative depending on the choice
of the parameter $\alpha$. Thus, we do not have log-convexity. (Although we have exponential decay, of course.)

\vspace{1.5cm}\noindent \textbf{Acknowledgement:} The research of A.S. has been funded by Deutsche Forschungsgemeinschaft (DFG) through the Collaborative
Research Center SFB 1114 ‘Scaling Cascades in Complex Systems’ (Project
No. 235221301), Subproject C05 ‘Effective models for materials and
interfaces with multiple scales’.

\footnotesize

\newcommand{\etalchar}[1]{$^{#1}$}
\def\cprime{$'$}











\end{document}